\documentclass[12pt, a4paper]{article}

\usepackage{amssymb}
\usepackage{amsmath}
\usepackage{amsthm}

  \newtheorem{theorem}{Theorem}

    \newtheorem{lemma}{Lemma}    
    \newtheorem{proposition}{Proposition}

    \newcommand\vu{{\bf u}}
    \newcommand\vv{{\bf  v}}
    \newcommand\vw{{\bf  w}}
    \newcommand\vb{{\bf b}}
    
    \newcommand\vN{{\bf  \nabla}}
    \newcommand\ddh{--}
    
    \DeclareMathOperator{\Div}{div}
    
\begin{document}

\title{Calder\'on splitting and weak solutions for Navier\ddh Stokes equations with
initial data in   weighted $L^p$ spaces.}
\author{  Pierre Gilles Lemari\'e-Rieusset\footnote{LaMME, Univ Evry, CNRS, Universit\'e Paris-Saclay, 91025, Evry, France}\footnote{ emeritus professor, (Univ Evry)} \footnote{e-mail : pierregilles.lemarierieusset@univ-evry.fr}}
\date{}\maketitle

\begin{abstract}
We show the existence of global weak solutions of the 3D Navier\ddh Stokes equations with initial velocity
in the weighted spaces  $L^p_{\Phi_\gamma}=L^p(\mathbb{R}^3, \Phi_\gamma(x)\, dx)$, where $2<p<+\infty$, $0<\gamma<2$ and $\Phi_\gamma(x)=\frac 1{(1+\vert x\vert^2)^{\frac\gamma 2}}$, using Calder\'on splitting  $L^p_{\Phi_\gamma}\subset L^2_{\Phi_2}+L^r$ (with some $r\in (3,+\infty)$) and   energy controls in $L^2_{\Phi_2}$. \end{abstract}
 


\section*{Introduction}
We consider the Navier\ddh Stokes equations on $(0,T)\times\mathbb{R}^3$
\begin{equation}\label{NSEq1} \left\{ \begin{split} \partial_t\vu=&\Delta\vu-\vu\cdot\vN\vu-\vN p
\\\Div\,\vu=&0 \\ \lim_{t\rightarrow 0} \vu(t,.)=&\vu_0\end{split}\right.\end{equation}
when the non-linearity $\vu\cdot\vN \vu$  is rewritten as $\mathrm{div}\, (\vu\otimes\vu)$ and the pressure $p$ is eliminated due do the Leray projection operator $\mathbb{P}$, rewritting $\vu\cdot\vN\vu+\vN p$ as $\mathbb{P}\Div(\vu\otimes\vu))$.
 
 We thus shall consider distribution $\vu$ such that
 $\vu\in L^2((0,T), L^2(\mathbb{R}^3,\frac {dx}{(1+\vert x\vert^2)^2} ))$, we first prove that $\mathbb{P}\Div(\vu\otimes\vu))$ is well defined as a distribution on $(0,T)\times\mathbb{R}^3$; assuming moreover that $\vu$ satisfies the  Navier\ddh Stokes equations 
 \begin{equation}\label{NSEq2} \left\{ \begin{split} \partial_t\vu=&\Delta\vu -\mathbb{P}\Div(\vu\otimes\vu))
\\\Div\,\vu=&0 \end{split}\right.\end{equation} we shall prove that $\vu$ may be defined as a continuous in time distribution $t\in [0,T]\mapsto \vu(t,.)\in \mathcal{S}'(\mathbb{R}^3)$, (more precisely, that $t\mapsto \frac 1{(1+\vert x\vert^2)^2} \vu(t,.)$ is continuous from $[0,T]$ to $ H^{-4}(\mathbb{R}^3)$), so that the initial value condition
\begin{equation}  \lim_{t\rightarrow 0} \vu(t,.)=\vu_0\end{equation} is meaningful. Those results are described in Proposition \ref{prop:proj} and Theorem \ref{theo:stokes}.

 There are many  examples of solutions such that 
 $\vu\in L^2((0,T), L^2(\mathbb{R}^3,\frac {dx}{(1+\vert x\vert^2)^2} ))$. A special subclass of such solutions are the solutions such that
 \begin{equation}\label{L2L2-uloc} \sup_{x_0\in\mathbb{R}^3} \iint_{(0,T)\times B(x_0,1)} \vert \vu (s,y)\vert^2\, ds\, dy<+\infty.\end{equation}  In this subclass, one can find
 \begin{itemize}
 \item[$\bullet$] Kato's mild solutions in Lebesgue spaces $\vu\in L^\infty((0,T), L^p)$ with $3\leq p\leq +\infty$ \cite{Kat84}
 corresponding to the Cauchy initial value problem with $\vu_0\in L^p$
 \item[$\bullet$]  more generally, Koch and Tataru's mild solutions \cite{KoT01} 
 corresponding to the Cauchy initial value problem with $\vu_0\in \mathrm{bmo}^{-1}$
 \item[$\bullet$]  Leray's weak solutions  $\vu\in L^\infty((0,T), L^2)$ \cite{Ler34}
 corresponding to the Cauchy initial value problem with $\vu_0\in L^2$
 \item[$\bullet$]  Calder\'on's weak solutions  $\vu\in L^\infty((0,T), L^2+L^\infty)$ \cite{Cal90}
 corresponding to the Cauchy initial value problem with $\vu_0\in L^p$ with $p\geq 2$
 \item[$\bullet$]  Lemari\'e-Rieusset's weak solutions  $\vu\in L^\infty((0,T), L^2_{\rm uloc})$ \cite{Lem99}
 corresponding to the Cauchy initial value problem with $\vu_0\in L^2_{\rm uloc}$
 \end{itemize}
 
 More recently, Fern\'andez-Dalgo \&  Lemari\'e-Rieusset \cite{FLR20} and Bradshaw, Kucavica \& Tsai \cite{BKT22} considered  the Cauchy initial value problem with $\vu_0\in L^2(\mathbb{R}^3, \frac{dx}{1+\vert x\vert^2})$. Their solutions belong to $L^\infty((0,T),  L^2(\frac{dx}{1+\vert x\vert^2}))$ and don't satisfy  condition
 (\ref{L2L2-uloc}).
 
 Other weak solutions which don't satisfy  condition
 (\ref{L2L2-uloc}) are the statistically homogeneous solutions of 
 Fursikov   and Vi\v{s}ik \cite{FuV88}, which are proved to belong to $L^\infty((0,T),  L^2(\frac{dx}{(1+\vert x\vert^2)^{\gamma/2}}))$ for every $\gamma>3$.
 
 Our main goal in this paper is to prove the following result:
 
 \begin{theorem}\label{theo:weightLp}  Let $2< p<+\infty$ and $\gamma\in (0,2)$.  Let $\vu_0\in L^p(\mathbb{R}^3, \frac{dx}{1+\vert x\vert^2})$ with $\mathrm{div}\, \vu_0=0$. Then the Navier\ddh Stokes equations
  \begin{equation}\label{NSEq3} \left\{ \begin{split} \partial_t\vu=&\Delta\vu -\mathbb{P}\Div(\vu\otimes\vu))
\\\Div\,\vu=&0  \\ \lim_{t\rightarrow 0} \vu(t,.)=&\vu_0\end{split}\right.\end{equation} have a solution on $(0,+\infty)\times\mathbb{R}^3$ with, for every $0<T<+\infty$, 
$\vu\in L^\infty((0,T), L^2(\frac{dx}{1+\vert x\vert^2}))+ L^\infty((0,T), L^r)$ for some $r\in (3,+\infty)$.
 \end{theorem}
    The limit case $L^\infty((0,T), L^2(\frac{dx}{1+\vert x\vert^2}))$ corresponds to the results of Fern\'andez-Dalgo \&  Lemari\'e-Rieusset \cite{FLR20} and Bradshaw, Kucavica \& Tsai \cite{BKT22}. The limit case $L^\infty((0,T),L^r)$ corresponds to Kato's mild solutions in $L^r$  \cite{Kat84}  (existence of mild solutions  is known only for a finite time $T\approx \frac 1{\|\vu_0\|_r^{\frac{2r}{r-3}}}$).  In order to deal with the case $2<p<+\infty$, we shall use Calder\'on's method  \cite{Cal90} and  split the initial value in a sum of two vector fields corresponding to the limit cases which we know how to deal with. 
    
\section{Weighted Lebesgue and Sobolev spaces}\label{sec:weighted}
Throughout the paper, we shall deal with weights $\Phi_\gamma(x)=\Phi(x)^\gamma$, $\gamma\in\mathbb{R}$, where $$\Phi(x)=\frac 1{(1+\vert x\vert^2)^{1/2}}\approx \frac 1{1+\vert x\vert}.$$ We will work in weighted spaces:
\begin{itemize}
\item[$\bullet$] weighted Lebesgue spaces $L^p(\Phi_\gamma\, dx)$ ($1\leq p\leq +\infty$) with
$$ \|u\|_{L^p(\Phi_\gamma\, dx)}=\|\Phi^{\frac\gamma p} u\|_p.$$
\item[$\bullet$] weighted Sobolev spaces $H^s(\Phi_\gamma\, dx)$ ($s\in\mathbb{R}$) with
$$ \|u\|_{H^s(\Phi_\gamma\, dx)}=\|\Phi^{\frac\gamma 2} u\|_{H^s}.$$
\end{itemize} $\mathcal{D}(\mathbb{R}^3)$ is dense in $L^p(\Phi_\gamma\, dx)$ for $1\leq p<+\infty$ and in $H^s(\Phi_\gamma\, dx)$. We have, for $u,v\in \mathcal{D}$ (and $p>1$)
\[ \vert \int uv\, dx\vert \leq \|u\|_{L^p(\Phi_\gamma dx)} \|v\|_{L^{\frac p{p-1}} (\Phi _{-\frac \gamma{p-1}}\, dx)} \] and \[  \vert \int uv\, dx\vert \leq \|u\|_{H^s(\Phi_\gamma dx)} \|v\|_{H^{-s} (\Phi _{- \gamma}\, dx)}.  \]

For $\gamma\leq \delta$ and $\sigma\leq s$, we have the following obvious  continuous embeddings:
\[ L^p(\Phi_\gamma\, dx)\subset L^p(\Phi_\delta \, dx)\text{ and } H^s(\Phi_\gamma\, dx)\subset H^\sigma(\Phi_\gamma \, dx).\]
Using the H\"older inequality, we find as well the following embedding
\[ L^p(\Phi_\gamma\, dx)\subset L^q(\Phi_\delta\, dx)\text{ for } q<p \text{ and } \frac\delta q> \frac \gamma p+ \frac{ 3(p-q)}{pq}.\]
Using the Sobolev inequalities, we find that, for $0\leq s<\frac 3 2$,  $\frac 1 q= \frac 1 2-\frac s 3$ and $\frac 1 r=\frac 1 2+\frac s 3$,  we have

\[ H^s(\Phi_\gamma\, dx)\subset L^q(\Phi_{\frac q 2 \gamma}\, dx) \text{ and }  L^r(\Phi_{\frac r 2 \gamma}\, dx)  \subset H^{-s}(\Phi_\gamma\, dx) .\] Similarly, for $s>\frac 3 2$, we have
\[ L^1(\Phi_{\frac  \gamma 2}\, dx)  \subset H^{-s}(\Phi_\gamma\, dx) .\] 

We state two further estimates:

\begin{lemma}\label{le:diff} Let $u\in H^s(\Phi_\gamma\,  dx)$. Then
\[ \|\vN u\|_{H^{s-1}(\Phi_\gamma \, dx)} \leq C \|  u\|_{H^{s}(\Phi_\gamma \, dx)} . \] If $s\geq 0$,
\[  \|  u\|_{H^{s}(\Phi_\gamma \, dx)}\approx \|u\|_{L^2(\Phi_\gamma\, dx)}+ \|\vN u\|_{H^{s-1}(\Phi_\gamma \, dx)} .\]
\end{lemma}
\begin{proof} We have
$$ \|\vN u\|_{H^{s-1}(\Phi_\gamma\, dx)}=\|\Phi^{\frac\gamma 2} \vN u\|_{H^{s-1}}\leq \| \vN (\Phi^{\frac \gamma 2} u)\|_{H^{s-1}}+ \frac{\vert\gamma \vert}2 \| \Phi^{\frac\gamma 2} u \frac{\vN \Phi}{\Phi}\|_{H^{s-1}}.$$
Since $ \frac{\vN \Phi}{\Phi}$ is bounded with all its derivatives, we find that 
\[ \| \Phi^{\frac\gamma 2} u \frac{\vN \Phi}{\Phi}\|_{H^{s-1}}\leq C \| \Phi^{\frac\gamma 2} u  \|_{H^{s-1}}\leq C \| \Phi^{\frac\gamma 2} u\|_{H^{s}}\] and we easily conclude since $ \| \vN(\Phi^{\frac\gamma 2} u ) \|_{H^{s-1}}\leq  \| \Phi^{\frac\gamma 2} u\|_{H^{s}}$.

If $s\geq 0$, we have \[ \|u\|_{H^s(\Phi_\gamma\, dx)}=\|\Phi^{\frac\gamma 2} u\|_{H^s}\approx  \|\Phi^{\frac\gamma 2} u\|_{2}+ \|\vN (\Phi^{\frac\gamma 2}u)\|_{H^{s-1}}.\]
We have
\[  \|\vN (\Phi^{\frac\gamma 2}u)\|_{H^{s-1}}\leq  \|\vN u\|_{H^{s-1}(\Phi_\gamma\, dx)} + C \|\Phi^{\frac\gamma 2} u\|_{H^{s-1}} \] with
\[  \|\Phi^{\frac\gamma 2} u\|_{H^{s-1}}\leq  \|\Phi^{\frac\gamma 2} u\|_{2} \] if $0\leq s\leq 1$ and 
\[  \|\Phi^{\frac\gamma 2} u\|_{H^{s-1}}\leq  \|\Phi^{\frac\gamma 2} u\|_{2}^{\frac 1 s} \|\Phi^{\frac\gamma 2} u\|_{H^{s}}^{1-\frac 1 s} \leq \frac 1{s\epsilon^s} \|\Phi^{\frac\gamma 2} u\|_{2}+ \epsilon^{\frac s{s-1}} (1-\frac 1 s) \|\Phi^{\frac\gamma 2} u\|_{H^{s}} \] if $s>1$ and $\epsilon>0$.
\end{proof}

\begin{lemma}\label{le:convol}
For $1\leq p\leq +\infty$ and $0<t\leq 1$ we have
\[ \| \int \frac 1{(\sqrt t+\vert x-y\vert)^4} f(y)\, dy\|_{L^p(\Phi_4\, dx)}\leq C t^{-\frac 1 2} \|f\|_{L^p(\Phi_4\, dx)}.\]\end{lemma} 
\begin{proof} This is obvious for  $p=+\infty$  since $\int \frac 1{(\sqrt t+\vert x-y\vert)^4}\,  dy= C  t^{-\frac 1 2}$.

 For $ p=1$,  we have to prove that 
 \[ \| \frac 1 {(1+\vert x\vert)^{ 4 } }\int \frac 1{(\sqrt t+\vert x-y\vert)^4} (1+\vert y\vert)^{ 4 } g(y)\, dy\|_1\leq C \|g\|_1.\] As $(1+\vert y\vert)^{ 4 } \leq 16 ((1+\vert x\vert)^{ 4 }+ \vert x-y\vert^{ 4 })$, we have \[ \int \frac{(1+\vert y\vert^4)} {(1+\vert x\vert)^4 (\sqrt t+\vert x-y\vert)^4}\, dx  \leq 16 ( \int \frac {dx} {(1+\vert x\vert)^4} + \int \frac {dx} {(\sqrt t+\vert x-y\vert)^4} )\leq C t^{-\frac 1 2}\] and we conclude by Fubini.
 
 For $1<p<+\infty$, we conclude by interpolation.
\end{proof}
Remark: we have as well 
\[ \| \int \frac 1{(\sqrt t+\vert x-y\vert)^4} f(y)\, dy\|_{L^p(\Phi_\gamma\, dx)}\leq C t^{-\frac 1 2} \|f\|_{L^p(\Phi_\gamma\, dx)}\] for $0\leq \gamma\leq 4$, $1\leq p\leq +\infty$ and $0<t\leq 1$ (by interpolation between $L^p(dx)$ and $L^p(\Phi_4\, dx)$).\\

Our next result deals with the Leray projection operator acting on the divergence of a tensor in $L^1(\Phi_4\, dx)$:

\begin{proposition}[Leray projection]\label{prop:proj} $ $ \\  Let $\mathbb{F} \in  L^1(\mathbb{R}^3, \Phi_4 \, {dx} )$. Then there exists a unique pair $(\vb_1, \vb_2)$ such that $$\Div\, (\mathbb{F})=\vb_1+\vb_2$$ with
\begin{itemize}
\item[$\bullet$] $  \vb_1\in H^\sigma(\Phi_\gamma dx)$ and $\Div\, \vb_1=0$
\item[$\bullet$]  $  \vb_2\in H^\sigma(\Phi_\gamma dx)$ and $\vN\wedge  \vb_2=0$
\item[$\bullet$] $\lim_{\tau\rightarrow +\infty} e^{\tau\Delta}\vb_2=0$ in $\mathcal{S}'(\mathbb{R}^3)$.
\end{itemize} for $\gamma>7$ and $\sigma<-\frac 5 2$.

$\vb_1$ is called the Leray projection of  $\mathrm{div}\, (\mathbb{F})$ and we write $$\vb_1=\mathbb{P}\Div(\mathbb{F}).$$ \end{proposition}
\begin{proof} Uniqueness is obvious: if $\Div  \mathbb{F}=\vb_1+\vb_2=\vb'_1+\vb'_2 $ and if $\vb=\vb_2-\vb'_2$, then
\[\Delta \vb =\vN\wedge(\vN\wedge(\vb_2-\vb'_2))-\vN(\Div(\vb_1-\vb_1'))=0 \] while
\[ \vb=\lim_{\tau\rightarrow +\infty} -\int_0^\tau e^{s\Delta}\Delta\vb\, ds=0.\]

We now construct $\vb_2$. We want to have \[\Delta \vb_2=\vN(\Div\vb_2)=\vN(\Div( \Div \mathbb{F}))=\vN(\sum_{i=1}^3 \sum_{j=1}^3\partial_i\partial_j F_{i,j})\] and
\[\vb_2=-\int_0^{+\infty} e^{s\Delta}\vN(\Div( \Div \mathbb{F}))\, ds.\] Let $\theta\in\mathcal{D}$ such that $0\leq\theta\leq 1$, $\theta$ is supported in the ball $B(0,2)$ and $\sum_{k\in\mathbb{Z}^3} \theta(x-k)=1$.  For $\sigma\in\mathbb{R}$ and $\gamma\in\mathbb{R}$, we write
\[ \|\vb_2\|_{H^\sigma(\Phi_\gamma\, dx)} \leq \sum_{k\in\mathbb{Z}^3} \sum_{j\in\mathbb{Z}^3} \| \frac{1}{(1+\vert x\vert^2)^{\frac\gamma 2}} \theta(x-k) \int_0^{+\infty} e^{s\Delta} \vN \Div(\Div(\theta(x-j)\mathbb{F}))\, ds\|_{H^\sigma}.\]
We take $\gamma>7$ and $\sigma<-\frac 5 2$. When $\vert j-k\vert \geq 8$, we write
\begin{equation*}\begin{split} \| \frac{1}{(1+\vert x\vert^2)^{\frac\gamma 2}} \theta(x-k)& \int_0^{+\infty} e^{s\Delta} \vN \Div(\Div(\theta(\cdot-j)\mathbb{F}))\, ds\|_{H^\sigma}\\ \leq  C \|  \frac{1}{(1+\vert x\vert^2)^{\frac\gamma 2}} &\theta(x-k) \int_0^{+\infty} e^{s\Delta} \vN \Div(\Div(\theta(x-j)\mathbb{F}))\, ds\|_1
\\ \leq   C' \|\frac 1{(1+\vert x\vert^2)^2} \theta(x-k)  &\int(\int_0^{+\infty} \vert (\vN\otimes\vN)\vN(W_s(x-y))\vert\, ds) \theta(y-j) \vert \mathbb{F}(y)\vert \, dy\\ =  C'' \|  \frac{1}{(1+\vert x\vert^2)^{ 2}} &\theta(x-k) \int \frac 1{\vert x-y\vert^4} \theta(y-j)\vert \mathbb{F}(y)\vert \, ds\|_1 \\ \leq   C''' \|  \frac{1}{(1+\vert x\vert^2)^{ 2}} &\theta(x-k) \int \frac 1{(1+\vert x-y\vert^2)^2}  \theta(y-j) \vert \mathbb{F}(y)\vert \, ds\|_1
\end{split}\end{equation*} with
\begin{equation*}\begin{split} \sum_{k\in\mathbb{Z}^3}\sum_{j\in\mathbb{Z}^3}  \|  \frac{1}{(1+\vert x\vert^2)^{ 2}}& \theta(x-k) \int \frac 1{(1+\vert x-y\vert^2)^2}  \theta(y-j)\vert \mathbb{F}(y)\vert \, ds\|_1
\\ &=  \|  \frac{1}{(1+\vert x\vert^2)^{ 2}}  \int \frac 1{(1+\vert x-y\vert^2)^2} \vert \mathbb{F}(y)\vert \, ds\|_1
\\ \leq & C' \|\mathbb{F}\|_{L^1(\Phi_4\,  dx)}
\end{split}\end{equation*}  (by Lemma \ref{le:convol}).

For $\vert j-k\vert <8$, we remark that $\frac {(1+\vert k\vert)^\gamma}  {(1+\vert x\vert^2)^{ \frac \gamma 2}} \theta(x-k) $ is smooth and bounded (with all its derivatives) independently from $k$, so that \begin{equation*} \begin{split}. \| \frac{1}{(1+\vert x\vert^2)^{\frac\gamma 2}} \theta(x-k) & \int_0^{+\infty} e^{s\Delta} \vN \Div(\Div(\theta(\cdot-j)\mathbb{F}))\, ds\|_{H^\sigma}\\ \leq& C \frac 1{(1+\vert k\vert)^\gamma} \|\mathbb{P}\Div (\theta(\cdot -j)\mathbb{F})\|_{H^\sigma}.\\ \leq &  C'  \frac 1{(1+\vert k\vert)^\gamma}  \|\theta(\cdot -j)\mathbb{F}\|_{H^{\sigma+1}}\\ \leq & C''  \frac 1{(1+\vert k\vert)^\gamma}\|\theta(\cdot -j)\mathbb{F}\|_{1}\\ \leq & C'''    \frac {(1+\vert j\vert)^4}{(1+\vert k\vert)^\gamma} \|\mathbb{F}\|_{L^1(\Phi_4\, dx)}.\end{split}
\end{equation*} We may conclude, as
\[ \sum_{k\in\mathbb{Z}^3}\sum_{j\in\mathbb{Z}^3, \vert j-k\vert<8}   \frac {(1+\vert j\vert)^4}{(1+\vert k\vert)^\gamma} <+\infty.\]

Finally, we study $e^{\tau\Delta}\vb_2$ when $\tau\rightarrow +\infty$.   We write
\begin{equation*}\begin{split} \|e^{\tau\Delta}\vb_2\|_{L^1(\Phi_4\, dx)}\leq    C \|\frac 1{(1+\vert x\vert^2)^2}  &\int(\int_0^{+\infty} \vert (\vN\otimes\vN)\vN(W_{s+\tau}(x-y))\vert\, ds) \vert \mathbb{F}(y)\vert \, dy\\\ \leq   C''\|  \frac{1}{(1+\vert x\vert^2)^{ 2}} & \int \frac 1{(\sqrt \tau+\vert x-y\vert^2)^2} \vert \mathbb{F}(y)\vert \, ds\|_1
\end{split}\end{equation*} with, for $\tau>1$, 
\begin{equation*} \frac{1}{(1+\vert x\vert^2)^{ 2}}  \frac 1{(\sqrt \tau+\vert x-y\vert^2)^2} \vert \mathbb{F}(y)\vert \leq  \frac{1}{(1+\vert x\vert^2)^{ 2}}  \frac 1{(1+\vert x-y\vert^2)^2} \vert \mathbb{F}(y)\vert 
\in L^1(\mathbb{R}^2).\end{equation*}  By dominated convergence, we find  that $\lim_{\tau\rightarrow +\infty}\|e^{\tau\Delta}\vb_2\|_{L^1(\Phi_4\, dx)}=0$.
\end{proof}

 \section{Weak solutions for the Stokes equations in $L^1((0,T), L^1(\mathbb{R}^3,\frac {dx}{(1+\vert x\vert)^4} ))$}\label{sec:weak}
 In this section, we consider the Stokes equations
 
  \begin{equation}\label{Stokes1} \left\{ \begin{split} \partial_t\vu=&\Delta\vu -\mathbb{P}\Div(\mathbb{F})
\\ \Div\vu=&0  \\ \lim_{t\rightarrow 0} \vu(t,.)=&\vu_0\end{split}\right.\end{equation}
where the tensor $\mathbb{F}=(F_{i,j})_{1\leq i,j\leq 3}$ belongs to $L^1((0,T), L^1(\mathbb{R}^3,\frac {dx}{(1+\vert x\vert)^4} ))$, $\Div(\mathbb{F})=\vb$ with $b_j=\sum_{i=1}^3\partial_i F_{i,j}$ and where the solution  $\vu$ belongs to $L^1((0,T), L^1(\mathbb{R}^3,\frac {dx}{(1+\vert x\vert)^4} ))$. [Remark: we don't study the existence of such a solution, we assume in this section that it exists.]

\begin{theorem}[Stokes equations in weighted Lebesgue space] \label{theo:stokes}$ $ \\ Let  $0<T<+\infty$ and $\mathbb{F} \in L^1((0,T), L^1(\mathbb{R}^3,\frac {dx}{(1+\vert x\vert)^4} ))$.  Let $\vu$ be a solution of the Stokes equation
\begin{equation}\label{stokes} \partial_t\vu=\Delta\vu-\mathbb{P}\Div(\mathbb{F})\end{equation} with $\vu \in L^1((0,T), L^1(\mathbb{R}^3,\frac {dx}{(1+\vert x\vert)^4} ))$.
Then  we have  \[\partial_t \vu\in L^1((0,T), H^{-4}(\Phi_8\, dx))\text{ and } \vu\in \mathcal{C}([0,T], H^{-4}(\Phi_8\, dx)).\]
In particular, if $\varphi\in\mathcal{D}(\mathbb{R})$ with $\varphi(0)=1$ and $\varphi(T)=0$, we have
\begin{equation}\label{initial}  \lim_{t\rightarrow 0} \vu(t,.)=-\int_{0<t<T} \partial_t (\varphi \vu)\, dt.\end{equation} Moreover, we have
\[ \Div(\vu(t,.))=e^{t\Delta} \Div(\vu(0,.)).\]
\end{theorem}
 \begin{proof}   As $   L^1(\Phi_4 \, dx)\subset  H^{-2}(\Phi_8 \, dx)$, we have  $\vu\in L^1((0,T), H^{-2}(\Phi_8\, dx))$ and $\Delta u\in L^1((0,T), H^{-4}(\Phi_8\, dx))$. From Proposition \ref{prop:proj}, we know that, for $\gamma>7$ and $\sigma<-5/2$, \[ \|\mathbb{P}(\Div \mathbb{F})(t,.)\|_{H^\sigma(\Phi_\gamma\, dx)}\leq C_{\gamma,\sigma} \|\mathbb{F}(t,.)\|_{L^1(\Phi_4\, dx)},\] hence $\mathbb{P}(\Div \mathbb{F})\in L^1((0,T), H^{-4}(\Phi_8\, dx))$.
 
 From $\partial_t\vu \in L^1((0,T), H^{-4}(\Phi_8\, dx))$ and $\vu \in  L^1((0,T), H^{-4}(\Phi_8\, dx))$, we conclude that $\vu\in  \mathcal{C}([0,T], H^{-4}(\Phi_8\, dx))$.
Finally, we write $$\partial_t\Div \vu=\Div\partial_t \vu=\Div\Delta\vu=\Delta \Div\vu$$ so that $\Div \vu(t,.)=e^{t\Delta}(\Div \vu(0,.))$.
 \end{proof}

\begin{theorem}[Solutions bounded in weighted Lebesgue space] \label{theo:stokes2}$ $ \\  Let  $0<T<+\infty$, $1<p<+\infty$ and $\mathbb{F} \in L^\infty((0,T), L^p(\mathbb{R}^3,\frac {dx}{(1+\vert x\vert)^4} ))$.  Let $\vu$ be a  distribution  defined on $(0,T)\times\mathbb{R}^3$. Then the following assertions are equivalent:\\
$\bullet$ assertion A1: $\vu$ is a 
solution of the Stokes equation
\begin{equation} \left\{\begin{split} &\partial_t\vu=\Delta\vu-\mathbb{P}\Div(\mathbb{F})\\& \Div \vu=0\end{split}\right. \end{equation} with $\vu \in L^\infty((0,T), L^p(\mathbb{R}^3,\frac {dx}{(1+\vert x\vert)^4} ))$.
\\
$\bullet$ assertion A2:  there exists a tempered distribution $\vu_0$ on $\mathbb{R}^3$ such that $\Div\vu_0=0$,  $\vu_0\in   L^p(\mathbb{R}^3,\frac {dx}{(1+\vert x\vert)^4} ))$  and
\[ \vu=e^{t\Delta}\vu_0 -\int_0^t e^{(t-s)\Delta} \mathbb{P}\Div(\mathbb{F})\, ds.\]
We then have $\lim_{t\rightarrow 0} \vu(t,.)=\vu_0$.\end{theorem}
\begin{proof} (A1) $\implies$ (A2): Let $\vu$ be a solution of the Stokes equation
\begin{equation}  \partial_t\vu=\Delta\vu-\mathbb{P}\Div(\mathbb{F})\end{equation} with $\vu \in L^\infty((0,T), L^p(\mathbb{R}^3,\frac {dx}{(1+\vert x\vert)^4} ))$.  As $\int \frac {dx}{(1+\vert x\vert)^4} <+\infty$, we have $\vu \in L^1((0,T), L^1(\mathbb{R}^3,\frac {dx}{(1+\vert x\vert)^4} ))$ and $\mathbb{F} \in L^1((0,T), L^1(\mathbb{R}^3,\frac {dx}{(1+\vert x\vert)^4} ))$. By Theorem \ref{theo:stokes}, we know that $\vu\in \mathcal{C}([0,T], H^{-4}(\Phi_8\, dx))$. We may then write $\lim_{t\rightarrow 0} \vu(t,.)=\vu_0$ as a strong limit in $H^{-4}(\Phi_8\, dx)$ but as well, since $\vu(t,.)$ is bounded in $L^p(\Phi_4\, dx)$, as a weak-* limit in  $L^p(\Phi_4\, dx)$.\\
(A2) $\implies$ (A1): From classical estimates on Oseen's tensor (see for instance section 4.5 in \cite{Lem24}), we have
\[ \vert e^{(t-s)\Delta} \mathbb{P}\Div(\mathbb{F})\vert\leq C \int \frac 1{(\sqrt{t-s}+\vert x-y\vert)^4} \vert \mathbb{F}(s,y)\vert\, dy\] and, by Lemma \ref{le:convol}, for $0<t<T$,
\begin{equation*}\begin{split} \|\int_0^t e^{(t-s)\Delta} \mathbb{P}\Div(\mathbb{F})\, ds\|_{L^p(\Phi_4\, dx)}\leq & C \int_0^t\max(\frac 1{\sqrt{t-s}},1) \| \mathbb{F}(s,.)\|_{L^p(\Phi_4\, dx)}\, ds\\ \leq& C'\sqrt t(1+\sqrt t)\|\mathbb{F}\|_{L^\infty ((0,T), L^p(\Phi_4\, dx))}.
\end{split}\end{equation*} Similarly, we write 
\[ \vert e^{t\Delta} \vu_0(x)\vert\leq C \int \frac {\sqrt t}{(\sqrt{t}+\vert x-y\vert)^4} \vert \vu_0(,y)\vert\, dy\] and 
\begin{equation*} \| e^{t\Delta} \vu_0\|_{L^p(\Phi_4\, dx)}\leq  C  \max(1,\sqrt t) \| \vu_0\|_{L^p(\Phi_4\, dx)}\tag*{\qedhere}
\end{equation*} \end{proof}

\section{Mollified equations} We want to find a weak solution to the Navier\ddh Stokes equations 
  \begin{equation}\label{NSEq5} \left\{ \begin{split} \partial_t\vu=&\Delta\vu -\mathbb{P}\Div(\vu\otimes\vu)
\\\Div\,\vu=&0  \\ \lim_{t\rightarrow 0} \vu(t,.)=&\vu_0\end{split}\right.\end{equation}   globally in time when $\vu_0$ is divergence free and $\vu_0\in L^p(\frac{dx}{(1+\vert x\vert)^\gamma})$ with $2< p<+\infty$ and $0<\gamma<2$.

 Following Leray \cite{Ler34}, we replace the Navier\ddh Stokes equations (\ref{NSEq5}) with the mollified equations  
  \begin{equation}\label{NSEq6} \left\{ \begin{split} \partial_t\vu_{\epsilon,\alpha}=&\Delta\vu_{\epsilon,\alpha} -\mathbb{P}\Div((\varphi_\epsilon*(\theta_\alpha\vu_{\epsilon,\alpha}))\otimes\vu_{\epsilon,\alpha})
\\\Div\,\vu_{\epsilon,\alpha}=&0  \\ \lim_{t\rightarrow 0} \vu_{\epsilon,\alpha}(t,.)=&\vu_{0}\end{split}\right.\end{equation} where $\varphi\in\mathcal{D}(\mathbb{R}^3)$ with $\varphi\geq 0$, $ \int\varphi(x)\, dx=1$, $\varphi(x)=0$ if $\vert x\vert>1$  and $\displaystyle \varphi_\epsilon(x)=\frac 1{\epsilon^3} \varphi(\frac x \epsilon)$ and where  $\theta\in\mathcal{D}(\mathbb{R}^3)$ with $0\leq\theta\leq 1$, $\theta(x)=1$ if $\vert x\vert\leq 1$, $\theta(x)=0$ if $\vert x\vert\geq 2$ and $\theta_\alpha(x)=\theta(\alpha x)$. 

We shall prove that equations (\ref{NSEq6}) have a solution \[\vu_{\epsilon,\alpha}\in \cap_{0<T<+\infty} L^\infty((0,T), L^p(\Phi_\gamma\, dx)).\] However, the control of $\vu_{\epsilon,\alpha}$ with respect to $\epsilon$  and $\alpha$ is not good enough when $({\epsilon,\alpha})$ goes to $(0,0)$: we find that
\[ \sup_{0<t<T} \|\vu_{\epsilon,\alpha}(t,.)\|_{L^p(\Phi_\gamma\, dx)}\leq C_{{\epsilon,\alpha},T, \vu_0} \text{ with } \lim_{{\epsilon,\alpha}\rightarrow (0,0)} C_{{\epsilon,\alpha},T,\vu_0}=+\infty.\]  Instead of $ L^p(\Phi_\gamma\, dx)$, we follow Calder\'on \cite{Cal90} and we shall work in $L^2(\Phi_2\, dx)+L^r$ (with $3<r<\infty$) and prove that, for every $T>0$, we
 have
\begin{equation}\label{eq:control1} \sup_{0<t<T} \|\vu_{\epsilon,\alpha}(t,.)\|_{L^2(\Phi_2\, dx)+L^r}\leq C_{T, \vu_0} <+\infty \end{equation} and
\begin{equation}\label{eq:control2}   \|\vu_{\epsilon,\alpha}(t,.)\|_{L^2((0,T), H^{\frac 1 2}(\Phi_4\, dx))}\leq C_{T, \vu_0} <+\infty.  \end{equation}  We shall see that estimates (\ref{eq:control1}) and (\ref{eq:control2}) are sufficient to grant some sequence $\vu_{{\epsilon_n,\alpha_n}}$ is weakly convergent  in $L^2((0,T), \Phi_4\, dx)$ to a solution $\vu$ of the Navier\ddh Stokes equations.

\subsection{First estimates in the norm of $L^p(\Phi_\gamma\, dx)$.}
In this section, $2\leq p<+\infty$ and $2\leq\gamma\leq 4$.

\begin{lemma}\label{le:troncatur} If $f\in L^p(\Phi_\gamma\, dx)$ then \[ \| \varphi_\epsilon*(\theta_\alpha f)\|_\infty\leq \|\varphi\|_{\frac p{p-1}} \frac 1{\epsilon ^{\frac 3 p}} (1+\frac 2\alpha)^{\frac \gamma p} \|f\|_{L^p(\Phi_\gamma\, dx)}.\] \end{lemma}

\begin{proof} Just write\[ \| \varphi_\epsilon*(\theta_\alpha f)\|_\infty\leq \|\varphi_\epsilon\|_{\frac p{p-1}} \|\theta_\alpha \Phi^{-\frac \gamma p}\|_\infty \|\Phi^{\frac \gamma p}f\|_p.\tag*{\qedhere}\] \end{proof}

\begin{lemma}\label{le:heat} If $\vu_0\in L^p(\Phi_\gamma\, dx)$, then $e^{t\Delta}\vu_0\in \mathcal{C}([0,+\infty[, L^p(\Phi_\gamma\, dx))$ and
\[ \|e^{t\Delta}\vu_0\|_{L^p(\Phi_\gamma\, dx)}\leq C  \max(1,\sqrt t) \|\vu_0\|_{L^p(\Phi_\gamma\, dx)}.\]
\end{lemma}

\begin{proof}  We write again
\[ \vert e^{t\Delta} \vu_0(x)\vert\leq C \int \frac {\sqrt t}{(\sqrt{t}+\vert x-y\vert)^4} \vert \vu_0(y)\vert\, dy\] and 
\begin{equation*} \| e^{t\Delta} \vu_0\|_{L^p(\Phi_\gamma\, dx)}\leq  C  \max(1,\sqrt t) \| \vu_0\|_{L^p(\Phi_\gamma\, dx)}.\end{equation*} Thus, for $0<T<+\infty$, convolution with the heat kernel is a bounded map from $L^p(\Phi_\gamma\, dx)$ to $L^\infty((0,T), L^p(\Phi_\gamma\, dx))$.

We then remark that $L^p(dx)$ is dense in $L^p(\Phi_\gamma\, dx)$, and that $W^{2,p}$ is dense in $L^p$. If $\vu_0\in W^{2,p}$, then, for $0\leq t\leq \tau$, we have
\[ \|e^{t\Delta}\vu_0-e^{\tau\Delta}\vu_0\|_{L^p(\Phi_\gamma\, dx)}\leq \|e^{t\Delta}\vu_0-e^{\tau\Delta}\vu_0\|_p\leq \int_t^\tau \| e^{s\Delta} \Delta \vu_0\|_p\, dx\leq (\tau-t)  \|\Delta \vu_0\|_p.\] Thus,  convolution with the heat kernel is a bounded map from $W^{2,p}$ to $\mathcal{C}([0,T], L^p(\Phi_\gamma\, dx))$ for the $W^{2,p}$ norm and  from $W^{2,p}$ to $L^\infty([0,T], L^p(\Phi_\gamma\, dx))$ for the $L^p(\Phi_\gamma\, dx)$ norm. Thus, it is a bounded map 
 from $L^p(\Phi_\gamma\, dx)$ to $\mathcal{C}([0,T], L^p(\Phi_\gamma\, dx))$. 
\end{proof}

\begin{lemma}\label{le:heat2} If  $\mathbb{F}\in L^\infty((0,T),L^p(\Phi_\gamma\, dx)$ (where $0<T<+\infty$), then $$\int_0^t e^{(t-s)\Delta}\mathbb{P}\Div\mathbb{F}\, ds\in \mathcal{C}([0,T], L^p(\Phi_\gamma\, dx))$$ and, for $0\leq t\leq T$, 
\[ \| \int_0^t e^{(t-s)\Delta}\mathbb{P}\Div\mathbb{F}\, ds\|_{L^p(\Phi_\gamma\, dx)}\leq C  \max(t,\sqrt t) \|\mathbb{F}\|_{L^\infty((0,T), L^p(\Phi_\gamma\, dx))}.\]
\end{lemma}

\begin{proof}  We write  
\[ \vert  \int_0^t e^{(t-s)\Delta}\mathbb{P}\Div\mathbb{F}\, ds\vert\leq C \int_0^t\int \frac 1{\sqrt{t-s}+\vert x-y\vert)^4} \vert  \mathbb{F}(s,y)\vert\, dy\] and 
\begin{equation*}\begin{split} \| \int_0^t e^{(t-s)\Delta}\mathbb{P}\Div\mathbb{F}\, ds\|_{L^p(\Phi_\gamma\, dx)}\leq&  C  \int_0^t \max(1,\frac 1{\sqrt{ t-s}}) \| \mathbb{F}(s,.)\|_{L^p(\Phi_\gamma\, dx)}\\ \leq& C'(t+\sqrt t) \|\mathbb{F}\|_{L^\infty((0,T), L^p(\Phi_\gamma\, dx))}..\end{split}\end{equation*} Thus, for $0<T<+\infty$, the operator $\mathcal{L}$ defined by
$$ \mathbb{F}\mapsto \mathcal{L}(\mathbb{F})= \int_0^t e^{(t-s)\Delta}\mathbb{P}\Div\mathbb{F}\, ds$$  is a bounded map from $L^\infty((0,T), L^p(\Phi_\gamma\, dx))$ to $L^\infty((0,T), L^p(\Phi_\gamma\, dx))$.

We then remark that $L^\infty((0,T),L^p(\Phi_\gamma\, dx)$   is embedded in $L^4((0,T),L^p(\Phi_\gamma\, dx))$  and that $L^\infty((0,T),W^{2,p})$ is dense in $L^4((0,T),L^p(\Phi_\gamma\, dx))$.  The operator $\mathcal{L}$   is a bounded map from $L^4((0,T), L^p(\Phi_\gamma\, dx))$ to $L^\infty((0,T), L^p(\Phi_\gamma\, dx))$:
\begin{equation*}\begin{split} \| \int_0^t e^{(t-s)\Delta}\mathbb{P}\Div\mathbb{F}\, ds\|_{L^p(\Phi_\gamma\, dx)}\leq&  C  \int_0^t \max(1,\frac 1{\sqrt{ t-s}}) \| \mathbb{F}(s,.)\|_{L^p(\Phi_\gamma\, dx)}\, ds\\ \leq& C \left( 
  \int_0^t \max(1,\frac 1{\sqrt{ t-s}}) ^{4/3}\, ds\right)^{3/4}  \|\mathbb{F}\|_{L^4((0,T), L^p(\Phi_\gamma\, dx))}\\ \leq& C' \max(t^{\frac 3 4}, t^{\frac 1 4}) \|\mathbb{F}\|_{L^4((0,T), L^p(\Phi_\gamma\, dx))}.\end{split}\end{equation*}

If $\mathbb{F}\in L^\infty((0,T),W^{2,p})$, then, for $0\leq t\leq T$, we have
\begin{equation*}\begin{split} \|\partial_t \int_0^t e^{(t-s)\Delta}\mathbb{P}\Div\mathbb{F}\, ds\|_{L^p(\Phi_\gamma\, dx)}\leq & \|\partial_t \int_0^t e^{(t-s)\Delta}\mathbb{P}\Div\mathbb{F}\, ds\|_{p}\\=&\| \mathbb{P}\Div\mathbb{F}(t,.)+\int_0^t e^{(t-s)\Delta}\mathbb{P}\Div\Delta\mathbb{F}\, ds\|_{p}\\ \leq&   C (\|\mathbb{F}(t,.)\|_{W^{2,p}} + \int_0^t  \frac 1{\sqrt{ t-s}} \|\Delta \mathbb{F}(s,.)\|_p\, ds)\\ \leq&  C'(1+\sqrt t) \|\mathbb{F}\|_{L^\infty((0,T), W^{2,p})}.
\end{split}\end{equation*}
  Thus, $\mathcal{L}$    is a bounded map from $ L^\infty((0,T),W^{2,p})$ to $\mathcal{C}([0,T], L^p(\Phi_\gamma\, dx))$ for the $ L^\infty((0,T),W^{2,p})$ norm and  from $ L^\infty((0,T),W^{2,p})$ to $L^\infty([0,T], L^p(\Phi_\gamma\, dx))$ for the $L^4((0,T), L^p(\Phi_\gamma\, dx))$ norm; we find that   $\mathcal{L}$    is a bounded map  
 from $L^\infty((0,T),L^p(\Phi_\gamma\, dx))$ to $\mathcal{C}([0,T], L^p(\Phi_\gamma\, dx))$. 
\end{proof}

\begin{proposition} \label{prop:mollif} Let $2\leq p<+\infty$, $2\leq\gamma\leq 4$ and $\vu_0\in L^p(\Phi_\gamma\ dx)$.  The mollified equations  
  \begin{equation}\label{NSEq7} \left\{ \begin{split} \partial_t\vu_{\epsilon,\alpha}=&\Delta\vu_{\epsilon,\alpha} -\mathbb{P}\Div((\varphi_\epsilon*(\theta_\alpha\vu_{\epsilon,\alpha}))\otimes\vu_{\epsilon,\alpha})
 \\ \lim_{t\rightarrow 0} \vu_{\epsilon,\alpha}(t,.)=&\vu_{0}\end{split}\right.\end{equation} have a unique (maximal) solution in $\mathcal{C}([0,T_{\epsilon,\alpha} [, L^p(\Phi_\gamma\, dx))$.
 
 If the maximal time of existence $T_{\epsilon,\alpha}$ is finite, then $\lim_{t\rightarrow T_{\epsilon,\alpha}} \|\vu_\epsilon(t,.)\|_{L^2(\Phi_4\, dx)}=+\infty$.
 
 There exists a constant $C_0>0$ such that
 \[  T_{\epsilon,\alpha}  >\min(1, \frac {\epsilon^3 \alpha^4}{C_0  \|\vu_0\|_{L^p(\Phi_\gamma\, dx)}^2 (2+\alpha)^4}).\]
\end{proposition}

\begin{proof}. We consider the fixed point problem in $\mathcal{C}([0,T], L^p(\Phi_\gamma\, dx))$
\[ \vu=e^{t\Delta}\vu_0-B_{\epsilon,\alpha}(\vu,\vu),\] where
\[ B_{\epsilon,\alpha}(\vv,\vw)=\int_0^t e^{(t-s)\Delta} \mathbb{P}\Div((\varphi_\epsilon*(\theta_\alpha \vv))\otimes\vw)\, ds.\].  We define $$R=\sup_{t\in [0,T]} \|e^{t\Delta}\vu_0\|_{L^p(\Phi_\gamma\, dx)}$$ and we want to prove that the map $\vu\mapsto e^{t\Delta}\vu_0-B_{\epsilon,\alpha}(\vu,\vu)$ is a contraction in $B_R=\{ \vu\in \mathcal{C}([0,T], L^p(\Phi_\gamma\, dx))\ /\ \sup_{t\in [0,T]} \| \vu(t,.)\|_{L^p(\Phi_\gamma\, dx)}\leq 2R\}$.

Fist, we use Lemma \ref{le:heat} and get that $R\leq  C_1  \max(1,\sqrt T) \|\vu_0\|_{L^p(\Phi_\gamma\, dx)}$.

As $L^p(\Phi_\gamma\, dx)\subset L^4(\Phi_2\, dx)$, we can use Lemmas \ref{le:troncatur}  and \ref{le:heat2} and find, for $\vv$, $\vw$ in $\mathcal{C}([0,T], L^p(\Phi_\gamma\, dx))$,  \[ \| \varphi_\epsilon*(\theta_\alpha \vv(t,.))\|_\infty\leq \|\varphi\|_2 \frac 1{\epsilon ^{\frac 3 2}} (\frac{2+\alpha}\alpha)^2 \|\vv(t,.)\|_{L^p(\Phi_\gamma\, dx)}\]  and, for $0<t<T$, 
\begin{equation*}\begin{split} & \| B_\epsilon(\vv,\vw)(t,.)\|_{L^p(\Phi_\gamma\, dx)}\\ & \leq C_2 (T+\sqrt T)  \frac 1{\epsilon ^{\frac 7 2}} (2+\epsilon)^2  \sup_{t\in [0,T]} \| \vv(t,.)\|_{L^p(\Phi_\gamma\, dx)}\sup_{t\in [0,T]} \| \vw(t,.)\|_{L^p(\Phi_\gamma\, dx)}. \end{split}\end{equation*} Thus, we will have a contraction in $B_R$ if
\[4C_2 (T+\sqrt T) \frac 1{\epsilon ^{\frac 3 2}} (\frac{2+\alpha}\alpha)^2 R<1, \] in particular if
\[4C_1C_2 \sqrt T (1+\sqrt T)^2 \frac 1{\epsilon ^{\frac 3 2}} (\frac{2+\alpha}\alpha)^2 \|\vu_0\|_{L^p(\Phi_\gamma\, dx)}<1. \]  This proves that
 that
 \[  T_{\epsilon,\alpha}  >\min(1, \frac {\epsilon^3 \alpha^4}{C_0  \|\vu_0\|_{L^p(\Phi_\gamma\, dx)}^2 (2+\alpha)^4}).\]
 
 If $T_{\epsilon,\alpha}$ is finite and $T<T_{\epsilon,\alpha}$, considering the initial value problem at initial time $T$, we find 
 \[  T_{\epsilon,\alpha}-T  >\min(1, \frac {\epsilon^3\alpha^4}{C_0  \|\vu_{\epsilon,\alpha}(T,.)\|_{L^p(\Phi_\gamma\, dx)}^2 (2+\alpha)^4}).\]
 Thus, $\lim_{T\rightarrow T_{\epsilon,\alpha}}  \|\vu_{\epsilon,\alpha}(T,.)\|_{L^p(\Phi_\gamma\, dx)}=+\infty$. We prove as well that $$\lim_{T\rightarrow T_{\epsilon,\alpha}}  \|\vu_{\epsilon,\alpha}(T,.)\|_{L^4(\Phi_2\, dx)}=+\infty.$$ If this was not the case, then the maximal existence time $\tilde T_{\epsilon,\alpha}$ for $\vu_{\epsilon,\alpha}$ in $\mathcal{C}([0,\tilde T_{\epsilon,\alpha}), L^2(\Phi_4\, dx)$ would satisfy $\tilde T_{\epsilon,\alpha}> T_{\epsilon,\alpha}$ and thus $\vu_{\epsilon,\alpha}$ would be bounded in $L^2(\Phi_4\, dx)$ on $[0,T_{\epsilon,\alpha})$.  For $0<T<t<T_{\epsilon,\alpha}$, we would have
 \begin{equation*}\begin{split} \|\vu_{\epsilon,\alpha}(t,.)\|_{L^p(\Phi_\gamma\, dx)}&\\ \leq \|e^{(t-T)\Delta}\vu_{\epsilon,\alpha}(T,.)\|_{L^p(\Phi_\gamma\, dx)}+&
\|  \int_T^t e^{(t-s)\Delta} \mathbb{P}\Div((\varphi_\epsilon*(\theta_\alpha \vu_{\epsilon,\alpha}))\otimes\vu_{\epsilon,\alpha})\, ds\|_{L^p(\Phi_\gamma\, dx)}\\ \leq C \max(1,\sqrt{T_{\epsilon,\alpha}-T})& \|\vu_{\epsilon,\alpha}(T,.)\|_{L^p(\Phi_\gamma\, dx)}\\+ C_2 (T_{\epsilon,\alpha}-T+\sqrt{ T_{\epsilon,\alpha}-T})  \frac 1{\epsilon ^{\frac 3  2}} (\frac{2+\alpha}{\alpha})^2&  \sup_{s\in [0,T_{\epsilon,\alpha}]} \| \vu_{\epsilon,\alpha}(s,.)\|_{L^2(\Phi_4\, dx)}\sup_{s\in [T,t]} \| \vu_{\epsilon,\alpha}(s,.)\|_{L^p(\Phi_\gamma\, dx)}.
 \end{split}\end{equation*}
 If $T$ is close enough to $T_{\epsilon,\alpha}$, so that  $$C_2 (T_{\epsilon,\alpha}-T+\sqrt{ T_{\epsilon,\alpha}-T})  \frac 1{\epsilon ^{\frac 3 2}} (\frac{2+\alpha}\alpha)^2  \sup_{s\in [0,T_{\epsilon,\alpha}]} \| \vu_{\epsilon,\alpha}(s,.)\|_{L^2(\Phi_4\, dx)}<\frac 1 2,$$
 we would get
 $$\sup_{T\leq t<T_{\epsilon,\alpha}}   \|\vu_{\epsilon,\alpha}(t,.)\|_{L^p(\Phi_\gamma\, dx)}\leq 2C \max(1,\sqrt{T_{\epsilon,\alpha}-T}) \|\vu_{\epsilon,\alpha}(T,.)\|_{L^p(\Phi_\gamma\, dx)}$$ in contradiction with  $\lim_{t\rightarrow T_{\epsilon,\alpha}}  \|\vu_{\epsilon,\alpha}(t,.)\|_{L^p(\Phi_\gamma\, dx)}=+\infty$.
\end{proof}

\subsection{Calder\'on's splitting.}
Now, for $2<p<+\infty$ and $0<\gamma<2$, we want to study the mollified equations when the initial data $\vu_0$ belongs to $L^p(\Phi_\gamma\, dx) \subset L^2(\Phi_4\, dx)$ with $\Div\vu_0=0$.
Following Calder\'on \cite{Cal90}, we will split the solution $\vu_{\epsilon,\alpha}$ as a sum $\vu_{\epsilon,\alpha}=\vv_{\eta,{\epsilon,\alpha}}+\vb_{\eta,{\epsilon,\alpha}}\in L^\infty((0,T_{(\eta)},  L^2(\Phi_2\, dx))+L^\infty((0,T_{(\eta)}), L^r)\subset  L^\infty((0,T_{(\eta)}), L^2(\Phi_4\, dx))$ for some $r\in (3,+\infty)$. The aim is to get a minoration of the existence time $T_{(\eta)}$ independent of $\epsilon$ and $\alpha$.
\begin{lemma}  \label{prop:split} Let $2< p<+\infty$ and $0<\gamma<2$. Let $r_0=2\frac{p-\gamma}{2-\gamma} $ and $\max(r_0,3)<r<+\infty$. Let  $\vu_0\in L^p(\mathbb{R}^3,\frac {dx}{(1+\vert x\vert)^2})$ with $\Div\vu_0=0$. Then, for every $\eta>0$ there exists $\vv_{0,\eta}$ and $\vb_{0,\eta}$ such that \[\vu_0=\vv_{0,\eta}+\vb_{0,\eta}\]  with \[ \vv_{0,\eta}\in L^2(\mathbb{R}^3,\frac {dx}{(1+\vert x\vert^2)}), \Div\vv_{0,\eta}=0  \] and 
  \[ \vb_{0,\eta}\in L^r(\mathbb{R}^3) ,  \Div\vb_{0,\eta}=0, \|\vb_{0,\eta}\|_r<\eta.\] 
\end{lemma}
\begin{proof} We have the interpolation result $L^p(\Phi_\gamma\, dx) =[L^2(\Phi_2\, dx), L^{r_0}]_{[1-\frac \gamma p]}$, so that $L^p(\Phi_\gamma\, dx) \subset L^2(\Phi_2\, dx)+ L^{r_0}$. Moreover, $$L^{r_0}=[L^2,L^r]_{[\frac{\frac1 2-\frac 1 {r_0}}{\frac 1 2-\frac 1 r}]} \subset [L^2,L^r]_{\frac{\frac1 2-\frac 1 {r_0}}{\frac 1 2-\frac 1 r},\infty}.$$
Let $\delta=\frac{\frac1 2-\frac 1 {r_0}}{\frac 1 2-\frac 1 r}$.  We may split $\vu_0$ in $\vu_0=\vu_1+\vu_2$ with $\vu_1\in L^2(\Phi_2\, dx)$ and $\vu_2\in L^{r_0}$ and, for every $A>0$, we may split $\vu_2 $ into $\vv_A+\vb_A$ with
\begin{equation*} \|\vv_A\|_{L^2(\Phi_2\, dx)}\leq C A^{\delta} \|\vu_2\|_{r_0}
\text{ and }  \|\vb_A\|_{r}\leq C A^{\delta-1} \|\vu_2\|_{r_0} .\end{equation*} Moreover, as $\Phi_\gamma$ is a Muckenhoupt weight in the class $\mathcal{A}_p$ and  $\Phi_2$ is a Muckenhoupt weight in the class $\mathcal{A}_2$, we may apply the Leray projection operator and write
$$ \vu_0=\mathbb{P}(\vu_0)=\mathbb{P}(\vu_1+\vv_A)+\mathbb{P}(\vb_A)$$  with \[ \mathbb{P}(\vu_1+\vv_A)\in L^2(\mathbb{R}^3,\frac {dx}{(1+\vert x\vert^2)}), \Div \mathbb{P}(\vu_1+\vv_A)=0  \] and 
  \[\mathbb{P}(\vb_A)\in L^r(\mathbb{R}^3) ,  \Div \mathbb{P}(\vb_A)=0,\| \mathbb{P}(\vb_A)\|_r< C A^{\delta-1} \|\vu_2\|_{r_0}.\] 
We conclude by taking $A$ large enough.
\end{proof}

For some $T$, we  then split the solution $\vu_{\epsilon,\alpha}$ (defined in $\mathcal{C}([0,T_{\epsilon,\alpha}), L^2(\Phi_4\, dx)))$ in $\vu_{\epsilon,\alpha}=\vv_{\eta,{\epsilon,\alpha}}+\vb_{\eta,{\epsilon,\alpha}}$ where $\vb_{\eta,{\epsilon,\alpha}} $ is a solution in $L^\infty((0,T), L^r)$ of $$\vb_{\eta,{\epsilon,\alpha}}=e^{t\Delta}\vb_{0,\eta}-B_{\epsilon,\alpha}(\vb_{\eta,{\epsilon,\alpha}},\vb_{\eta,{\epsilon,\alpha}})
$$ and $\vv_{\eta,{\epsilon,\alpha}}$ is a solution in  $\displaystyle \bigcap_{S<T_{\epsilon,\alpha}} L^\infty((0,\min(S,T), L^2(\Phi_4\, dx))$ of $$\vv_{\eta,{\epsilon,\alpha}}=e^{t\Delta}\vv_{0,\eta}-B_{\epsilon,\alpha}(\vb_{\eta,{\epsilon,\alpha}},\vv_{\eta,{\epsilon,\alpha}})-B_{\epsilon,\alpha}(\vv_{\eta,{\epsilon,\alpha}},\vb_{\eta,{\epsilon,\alpha}})-B_{\epsilon,\alpha}(\vv_{\eta,{\epsilon,\alpha}},\vv_{\eta,{\epsilon,\alpha}}).
$$

\subsection{Local estimates in the norm of $L^r$.}

We write $W_t$ for the heat kernel $W_t(x)=\frac 1{(4\pi t)^{\frac 3 2}} e^{-\frac{\vert x\vert^2}{4t}}$.
\begin{proposition} \label{prop:mollif2}  Let $r>3$. There exists a constant $C_1>0$ such that, for every $\epsilon>0$,  $\alpha>0$, $\eta>0$ and every $\vb_{0,\eta}\in L^r$ with $\|\vb_{0,\eta}\|_r\leq \eta$,  the mollified equations  
  \begin{equation}\label{NSEq8} \left\{ \begin{split} \partial_t\vb_{\eta,{\epsilon,\alpha}}=&\Delta\vb_{\eta,{\epsilon,\alpha}} -\mathbb{P}\Div((\varphi_\epsilon*(\theta_\alpha\vb_{\eta,{\epsilon,\alpha}}))\otimes  \vb_{\eta,{\epsilon,\alpha}})
 \\ \lim_{t\rightarrow 0}  \vb_{\eta,{\epsilon,\alpha}}(t,.)=&\vb_{0,\eta}\end{split}\right.\end{equation} have a unique  solution  on $(0,T_{[\eta]})\times \mathbb{R}^3$ (with $T_{[\eta]}^{\frac 1 2-\frac 3{2r}}= \frac 1{C_1 \eta}$) such that 
 \begin{itemize}\item $\vb_{\eta,{\epsilon,\alpha}}\in\mathcal{C}([0, T_{[\eta]}], L^2(\Phi_4\, dx))$,
 \item $\sup_{0\leq t\leq T_{[\eta]}} \|\vb_{\eta,{\epsilon,\alpha}}(t,.)\|_r\leq 2 \eta$,
 \item $\sup_{0\leq t\leq T_{[\eta]}}  t^{\frac 1 2}\|\vN\otimes \vb_{\eta,{\epsilon,\alpha}}(t,.)\|_r\leq  2  \|\vN W_1\|_1 \eta$,
 \item $\sup_{0< t\leq T_{[\eta]}} t^{ \frac 3 {2r}}\| \vb_{\eta,{\epsilon,\alpha}}(t,.)\|_\infty\leq 2  \|  W_1\|_{\frac r{r-1}} \eta$.
 \end{itemize}
\end{proposition}

\begin{proof} First, we remark that $L^r(\mathbb{R}^3)\subset L^2(\Phi_4\, dx)$. We have the obvious results for the heat kernel operating on $L^\infty$:
\[ \|e^{t\Delta}f\|_{L^r(dx)} =\|W_t*f\|_r\leq \|W_t\|_1\|f\|_r=\|f\|_r,\] 
\[ \|\vN e^{t\Delta}f\|_{L^r(dx)} =\|\vN W_t*f\|_r\leq \|\vN W_t\|_1\|f\|_r=\frac 1{\sqrt t}\|\vN W_1\|_1\|f\|_r\] and
\[ \|  e^{t\Delta}f\|_{L^\infty(dx)} =\|  W_t*f\|_\infty\leq \|  W_t\|_{\frac r{r-1}}\|f\|_r=\frac 1{  t^{\frac 3{2r}}}\|  W_1\|_{\frac r{r-1}}\|f\|_r.
\]

We consider the fixed point problem  
\[ \vb=e^{t\Delta}\vb_{0,\eta}-B_{\epsilon,\alpha}(\vb,\vb).\]  We want to prove that the map $\vb\mapsto e^{t\Delta}\vb_{0,\eta}-B_{\epsilon,\alpha}(\vb,\vb)$ is a contraction in \begin{equation*}\begin{split}B_\eta=\{ \vb\in \mathcal{C}([0,T], L^2(\Phi_4\, dx))\ /\ &\sup_{t\in [0,T]} \| \vb(t,.)\|_r\leq 2\eta,\\\ & \sup_{0<t\leq T} \sqrt t \|\vN\otimes \vb(t,.)\|_r\leq 2 \|W_1\|_1 \eta,\\& \sup_{0<t\leq T}  t^{\frac 3{2r}}   \|  \vb(t,.)\|_\infty\leq 2 \|W_1\|_{\frac r{r-1}} \eta
\}.\end{split}\end{equation*}
Let $\vv$, $\vw\in  \mathcal{C}([0,T], L^2(\Phi_4\, dx))$ with 
\begin{itemize}
\item[$\bullet$]  $\sup_{0<t<T} \|\vv(t,.)\|_r<+\infty$, \item[$\bullet$] $\sup_{0<t<T} \|\vw(t,.)\|_r<+\infty$,
\item[$\bullet$]  $\sup_{0<t<T} t^{\frac 3{2r}}  \|\vv(t,.)\|_\infty<+\infty$, \item[$\bullet$] $\sup_{0<t<T}t^{\frac 3{2r}}   \|\vw(t,.)\|_\infty<+\infty$,
\item[$\bullet$] $\sup_{0<t<T} \sqrt t \|\vN\otimes\vv(t,.)\|_r<+\infty$,  
\item[$\bullet$] $\sup_{0<t<T} \sqrt t \|\vN\otimes\vw(t,.)\|_r<+\infty$. 
\end{itemize}

We have the inequalities
 \begin{equation*}\begin{split}\|B_{\epsilon,\alpha}(\vv,\vw)\|_r\leq & C \int_0^t \frac 1{\sqrt{t-s}} \|\varphi_\epsilon*(\theta_\alpha \vv(s,.))\|_r  \|\vw(s,.)\|_\infty\, ds
 \\\leq& C \int_0^t \frac 1{\sqrt{t-s}}\frac 1 {s^{\frac 3 {2r}}}\, ds \sup_{0<t<T} \|\vv(s,.)\|_r \sup_{0<s<T}s^{\frac 3{2r}}   \|\vw(s,.)\|_\infty\\ \leq &C' t^{\frac 1 2(1-\frac 3 r)} \sup_{0<t<T} \|\vv(s,.)\|_r \sup_{0<s<T}s^{\frac 3{2r}}   \|\vw(s,.)\|_\infty,
  \end{split}\end{equation*}
 \begin{equation*}\begin{split} \|B_{\epsilon,\alpha}(\vv,\vw)\|_\infty\leq & C \int_0^t \frac 1{\sqrt{t-s}} \|\varphi_\epsilon*(\theta_\alpha \vv(s,.))\|_\infty \|\vw(s,.)\|_\infty\, ds
\\\leq& C \int_0^t \frac 1{\sqrt{t-s}}\frac 1 {s^{\frac 3 {r}}}\, ds \sup_{0<t<T} s^{\frac 3{2r}} \|\vv(s,.)\|_\infty \sup_{0<s<T}s^{\frac 3{2r}}   \|\vw(s,.)\|_\infty\\ \leq &C't^{-\frac 3{2r}} t^{\frac 1 2(1-\frac 3 r)} \sup_{0<t<T} s^{\frac 3{2r}} \|\vv(s,.)\|_\infty \sup_{0<s<T}s^{\frac 3{2r}}   \|\vw(s,.)\|_\infty,
  \end{split}\end{equation*} and, for $1\leq j\leq 3$, remarking that $\|\partial_j\theta_\alpha\|_3=\|\partial_j\theta\|_3$,
 \begin{equation*}\begin{split}\|\partial_j B_{\epsilon,\alpha}(\vv,\vw)\|_r\leq & C \int_0^t \frac 1{\sqrt{t-s}} \|\varphi_\epsilon*(\theta_\alpha \vv(s,.))\|_\infty\| \partial_j\vw(s,.)\|_r\, ds\\&+ C\int_0^t \frac 1{\sqrt{t-s}} \| \varphi_\epsilon*(\theta_\alpha \partial_j\vv(s,.))\|_r \| \vw(s,.)\|_\infty\, ds\\&+ C\int_0^t \frac 1{(t-s)^{1-\frac 3{2r}}}   \| \varphi_\epsilon* (\partial_j\theta_\alpha\, \vv(s,.))\|_3 \| \vw(s,.)\|_\infty\, ds
 \\ \leq C' t^{-\frac  3 {2r}}& (\sup_{0<t<T} s^{\frac 3{2r}} \|\vv(s,.)\|_\infty+\sup_{0<s<T}\sqrt s\|\vN\otimes\vv(s,.)\|_r) \\& (\sup_{0<s<T}s^{\frac 3{2r}}   \|\vw(s,.)\|_\infty+\sup_{0<s<T}\sqrt s\|\vN\otimes \vw(s,.)\|_r).
  \end{split}\end{equation*}
   
For the norm
$$ \|\vv \|_*=\max( \sup_{0<s<T} \|\vv(s,.)\|_r ,\sup_{0<s<T} \sqrt s \|\vN\otimes\vv(s,.)\|_r, \sup_{0<s<T}s^{\frac 3{2r}} \|\vv(s,.)\|_\infty),$$ we find
$$ \|B_{\epsilon,\alpha}(\vv,\vw)\|_* \leq C   T^{\frac 1 2-\frac 3{2r}} \|\vv\|_* \|\vw\|_*$$ (where $C$ doesn't depend on $\epsilon$ nor on $\alpha$),  which proves that $B_{\epsilon,\alpha}$ is a contraction on $B_\eta$  for $T^{\frac 1 2-\frac 3{2r}}\eta$ small enough.
\end{proof}

\subsection{Local estimates in the norm of $L^2(\Phi_2\, dx)$.}
We now want to estimate $\vv_{\eta, {\epsilon,\alpha}}(t,.)$ in  $L^2(\Phi_2\, dx)$ for $\leq t\leq T<\min(T_{\epsilon,\alpha}, T_{[\eta]})$. We have
\[ \vv_{\eta,{\epsilon,\alpha}}(t,.)=e^{t\Delta}\vv_{0,\eta} -\mathcal{L}(\mathbb{F}_{\eta,{\epsilon,\alpha}})\] where we defined  the operator $\mathcal{L}$  by
$$ \mathbb{F}\mapsto \mathcal{L}(\mathbb{F})= \int_0^t e^{(t-s)\Delta}\mathbb{P}\Div\mathbb{F}\, ds$$ and where
$$ \mathbb{F}_{\eta,{\epsilon,\alpha}}= (\varphi_\epsilon*(\theta_\alpha\vu_{\epsilon,\alpha}))\otimes\vu_{\epsilon,\alpha} -(\varphi_\epsilon*(\theta_\alpha\vb_{\eta,{\epsilon,\alpha}}))\otimes  \vb_{\eta,{\epsilon,\alpha}}. $$

\begin{lemma}\label{le:weight2}  $\vv_{\eta, {\epsilon,\alpha}}\in \mathcal{C}([0,T], L^2(\Phi_2\, dx))\cap L^2((0,T), H^1(\Phi_2\, dx))$ and  $\partial_t \vv_{\eta, {\epsilon,\alpha}}\in  L^2((0,T), H^{-1}(\Phi_2\, dx))$. In particular, we have
\begin{equation*} \begin{split}\int \vert \vv_{\eta, {\epsilon,\alpha}}(t,x)\vert^2 \Phi_2(x)\, dx-&\int \vert \vv_{0,\eta}(x)\vert^2 \Phi_2(x)\, dx \\ =& 2\int_0^t \int  \vv_{\eta, {\epsilon,\alpha}}(s,x)\cdot \partial_t \vv_{\eta, {\epsilon,\alpha}}(s,x) \ \Phi_2(x)\, ds\, dx \\ = 2\int_0^t \int  \vv_{\eta, {\epsilon,\alpha}}(s,x)\cdot&(\Delta\vv_{\eta, {\epsilon,\alpha}}(s,x) -\mathbb{P}\Div \mathbb{F}_{\eta,{\epsilon,\alpha}}(s,x))\ \Phi_2(x)\, ds\, dx.
\end{split}\end{equation*}\end{lemma}

\begin{proof} First, we recall that $\|f\|_{H^s(\Phi_2\, dx)}=\|\Phi f\|_{H^s}$. As $\Phi$ is bounded with all its derivatives, we have $\|\Phi f\|_{H^{s}}\leq C_s \|f\|_{H^s}$ for every $s\in\mathbb{R}$.   As $\Phi_2$ is a weight in the Muckenhoupt class $\mathcal{A}_2$ we have 
\[ \|e^{t\Delta}f\|_{L^2(\Phi_2\,  dx)}\leq \|\mathcal{M}_f\|_{L^2(\Phi_2\,  dx)}\leq C \|f\|_{L^2(\Phi_2\,  dx)}\] where $\mathcal{M}_f$ is the Hardy\ddh Littlewood maximal function of $f$.

  If $\vw_0\in L^2$, we have $e^{t\Delta}\vw_0\in  \mathcal{C}([0,T], L^2)\cap L^2((0,T), H^1)$ and $\partial_t(e^{t\Delta}\vw_0)\in L^2((0,T), H^{-1})$; in particular, $\Phi e^{t\Delta}\vw_0 \in L^2((0,T), H^1)$ and $\partial_t(\Phi e^{t\Delta}\vw_0) \in L^2((0,T), H^{-1})$. We thus have, for $0<t<T$,
\begin{equation*}\begin{split}\| \Phi  e^{t\Delta}\vw_0\|_2^2-\| \Phi   \vw_0\|_2^2=&2\int_0^t\int (\Phi e^{s\Delta}\vw_0)(s,x)\cdot \partial_t( \Phi e^{s\Delta}\vw_0)(s,x)\, ds\ dx\\=&2\int_0^t\int \Phi(x)^2 e^{s\Delta}\vw_0(s,x) \cdot\Delta   e^{s\Delta}\vw_0(s,x)\, ds\, dx
\\=& \int_0^t\int \Phi(x)^2( \Delta \vert e^{s\Delta}\vw_0(s,x) \vert^2-2 \vert \vN  e^{s\Delta}\vw_0(s,x)\vert^2)\, ds\, dx.
\end{split}\end{equation*} so that, as $\vert \Delta \Phi_2\vert\leq C \Phi_2$, 
\begin{equation*}\begin{split} 2 \int_0^t \| \vN  e^{s\Delta}\vw_0(s,.)\|_{L^2(\Phi_2\, dx)}^2\, ds&\\=- \| \Phi  e^{t\Delta}\vw_0\|_2^2+&\| \Phi   \vw_0\|_2^2+ \int_0^t\int  \vert e^{s\Delta}\vw_0(s,x) \vert^2 \Delta\Phi_2(x)\, ds\, dx\\ \leq& C \|\vw_0\|_{L^2(\Phi_2\, dx)}^2 (1+t).
\end{split}\end{equation*}   
By density of $L^2$ in $L^2(\Phi_2\, dx)$, we find that $\vw_0\mapsto e^{t\Delta}\vw_0$ is a bounded map from $L^2(\Phi_2\, dx)$ to  $\mathcal{C}([0,T], L^2(\Phi_2\, dx))\cap L^2((0,T), H^1(\Phi_2\, dx))$. Moreover, since $\partial_t e^{t\Delta}\vw_0=\Delta e^{t\Delta}\vw_0$, we have that $\partial_t e^{t\Delta}\vw_0\in L^2((0,T), H^{-1}(\Phi_2\, dx))$.

 A classical result on the heat kernel states that, if $\mathbb{F}\in L^2((0,T), L^2)$, then $\int_0^t e^{(t-s)\Delta} \mathbb{P}\Div\mathbb{F}\in  \mathcal{C}([0,T], L^2)\cap L^2((0,T), H^1)$ and thus  $\int_0^t e^{(t-s)\Delta} \mathbb{P}\Div\mathbb{F}\in  \mathcal{C}([0,T], L^2(\Phi_2\, dx))\cap L^2((0,T), H^1(\Phi_2\, dx))$. 

As we have
\begin{equation*}\begin{split} \|\mathbb{F}_{\eta,{\epsilon,\alpha}}(t,.)\|_2\leq &\| \varphi_\epsilon*(\theta_\alpha\vu_{\epsilon,\alpha})\|_\infty (\int_{\vert x\vert\leq \frac 2 \alpha +\epsilon} \vert \vu_{\epsilon,\alpha}(t,x)\vert^2\, dx)^{\frac 1 2} \\&+\| \varphi_\epsilon*(\theta_\alpha\vb_{\epsilon,\alpha})\|_\infty (\int_{\vert x\vert\leq \frac 2 \alpha +\epsilon} \vert \vb_{\epsilon,\alpha}(t,x)\vert^2\, dx)^{\frac 1 2},\end{split}\end{equation*}    we find that 
$$\mathcal{L}(\mathbb{F}_{\eta,{\epsilon,\alpha}})\in \mathcal{C}([0,T], L^2)\cap L^2((0,T), H^1)\subset  \mathcal{C}([0,T], L^2(\Phi_2\, dx))\cap L^2((0,T), H^1(\Phi_2\, dx)).$$ Moreover, $\partial_t \mathcal{L}(\mathbb{F}_{\eta,{\epsilon,\alpha}})=\Delta \mathcal{L}(\mathbb{F}_{\eta,{\epsilon,\alpha}}) -\mathbb{P}\Div \mathbb{F}_{\eta,{\epsilon,\alpha}}$ with
\[ \Delta \mathcal{L}(\mathbb{F}_{\eta,{\epsilon,\alpha}}) \in L^2((0,T), H^{-1}(\Phi_2\, dx))\]
and 
\[ \mathbb{P}\Div(\mathbb{F}_{\eta,{\epsilon,\alpha}}) \in L^2((0,T), H^{-1})\subset L^2((0,T), H^{-1}(\Phi_2\, dx)).\]

Finally, as $\Phi  \vv_{\eta, {\epsilon,\alpha}}\in L^2((0,T), H^1)$ and $\partial_t(\Phi  \vv_{\eta, {\epsilon,\alpha}})\in L^2((0,T), H^{-1})$, one has 
\begin{equation*}\| \Phi  \vv_{\eta, {\epsilon,\alpha}}(t,.)\|_2^2-\!\| \Phi  \vv_{\eta, {\epsilon,\alpha}}(0,.)\|_2^2\!=2\!\int_0^t\!\int (\Phi\vv_{\eta, {\epsilon,\alpha}})(s,x)\cdot \partial_t( \Phi\vv_{\eta, {\epsilon,\alpha}})(s,x)\, ds\ dx.\tag*{\qedhere}\end{equation*}
\end{proof}

Lemma \ref{le:weight2} will be a key ingredient for controlling the norm of $  \vv_{\eta, {\epsilon,\alpha}}$ in $ \mathcal{C}([0,T], L^2(\Phi_2\, dx))\cap L^2((0,T), H^1(\Phi_2\, dx))$ and the existence time of the solution  $  \vv_{\eta, {\epsilon,\alpha}}$  independently from $\epsilon$ and $\alpha$.

\begin{proposition} \label{prop:mollif3}  Let $2< p<+\infty$ and $0<\gamma<2$. There exists a constant $C_2>1$ such that, for every $\epsilon>0$, every $\alpha>0$, every $\eta>0$ and every $\vu_{0}\in L^p(\Phi_\gamma\, dx)$ with  $\Div \vu_0=0$, writing $\vu_0=\vb_{0,\eta}+\vv_{0,\eta}$ with $\vv_{0,\eta}\in  L^2(\Phi_2\, dx)$, $\Div \vv_{0,\eta}=0$ and $\|\vb_{0,\eta}\|_r<\eta$, $\Div\vb_{0,\eta}=0$, and writing $\vu_{\epsilon,\alpha}=\vv_{\eta,{\epsilon,\alpha}}+\vb_{\eta,{\epsilon,\alpha}}$ with $\vu_{\epsilon,\alpha}\in \mathcal{C}([0,T_{\epsilon,\alpha}), L^p(\Phi_\gamma\, dx)$ (as described in Proposition \ref{prop:mollif}) and $\vb_{\eta,{\epsilon,\alpha}}\in L^\infty((0,T_{[\eta]}), L^r)$ (as described in  Proposition \ref{prop:mollif2}),  the mollified equation
  \begin{equation}\label{NSEq10} \vv_{\eta,{\epsilon,\alpha}}=e^{t\Delta}\vv_{0,\eta}-B_{\epsilon,\alpha}(\vb_{\eta,{\epsilon,\alpha}},\vv_{\eta,{\epsilon,\alpha}})-B_{\epsilon,\alpha}(\vv_{\eta,{\epsilon,\alpha}},\vb_{\eta,{\epsilon,\alpha}})-B_{\epsilon,\alpha}(\vv_{\eta,{\epsilon,\alpha}},\vv_{\eta,{\epsilon,\alpha}})\end{equation} has a unique solution  on $(0,   T_{\eta,{\epsilon,\alpha}})\times \mathbb{R}^3$ such that 
 \begin{itemize}
 \item $ T_{\eta,{\epsilon,\alpha}}=\min(T_{\epsilon,\alpha}, \frac 1{C_2}T_{[\eta]},  \frac 1{C_2} \frac 1{1+ \max(1,\alpha)^6\|\vv_{0,\eta}\|_{L^2(\Phi_2\, dx)}^4})$
 \item $\vv_{\eta,{\epsilon,\alpha}}\in\mathcal{C}([0, T_{\eta,{\epsilon,\alpha}}], L^2(\Phi_2\, dx))$ with
\[\sup_{0\leq t\leq T_{\eta,{\epsilon,\alpha}}} \|\vv_{\eta,{\epsilon,\alpha}}(t,.)\|_{L^2(\Phi_2\, dx)}\leq 2  \| \vv_{0,\eta}\|_{L^2(\Phi_2\, dx)},\]
 \item $\vv_{\eta,{\epsilon,\alpha}}\in L^2((0, T_{\eta,{\epsilon,\alpha}}), H^1(\Phi_2\, dx))$ with
\[\|\vv_{\eta,{\epsilon,\alpha}}(t,.)\|_{L^2((0,T_{\eta,{\epsilon,\alpha}})(H^1(L^2(\Phi_2\, dx))}\leq C_2  \|\vv_{0,\eta}\|_{L^2(\Phi_2\, dx)}.\] \end{itemize}
 In particular, we have
 $$ T_{\epsilon,\alpha}\geq \min(\frac 1{C_2} (\frac 1{C_1\|\vb_{0,\eta}\|_r})^{\frac{2r}{r-3}},  \frac 1{C_2} \frac 1{1+ \max(1,\alpha)^6\|\vv_{0,\eta}\|_{L^2(\Phi_2\, dx)}^4}).$$
\end{proposition}
\begin{proof} We write 
 \begin{equation*}\begin{split} \partial_t\vv_{\eta,{\epsilon,\alpha}}=\Delta\vv_{\eta,{\epsilon,\alpha}}&-\mathbb{P}\Div((\varphi_\epsilon*(\theta_\alpha \vv_{\eta,{\epsilon,\alpha}}))\otimes\vb_{\eta,{\epsilon,\alpha}})-\mathbb{P}\Div ((\varphi_\epsilon*(\theta_\alpha \vb_{\eta,{\epsilon,\alpha}}))\otimes\vv_{\eta,{\epsilon,\alpha}})\\&-\mathbb{P}\Div((\varphi_\epsilon*(\theta_\alpha \vv_{\eta,{\epsilon,\alpha}}))\otimes\vv_{\eta,{\epsilon,\alpha}})\\=\Delta\vv_{\eta,{\epsilon,\alpha}}&-\mathbb{P}\Div((\varphi_\epsilon*(\theta_\alpha \vv_{\eta,{\epsilon,\alpha}}))\otimes\vb_{\eta,{\epsilon,\alpha}})-\mathbb{P}\Div ((\varphi_\epsilon*(\theta_\alpha \vb_{\eta,{\epsilon,\alpha}}))\otimes\vv_{\eta,{\epsilon,\alpha}})\\&- \Div((\varphi_\epsilon*(\theta_\alpha \vv_{\eta,{\epsilon,\alpha}})\otimes  \vv_{\eta,{\epsilon,\alpha}}))-\vN q_{\eta,\epsilon,\alpha}\end{split}\end{equation*} with
 \[ q_{\eta,\epsilon,\alpha}=\sum_{1\leq i\leq 3}\sum_{1\leq j\leq 3} R_iR_j ((\varphi_\epsilon*(\theta_\alpha \vv_{\eta,{\epsilon,\alpha},i}))\vv_{\eta,{\epsilon,\alpha},j}). \] This gives
  (as $\Div \vv_{\eta,\epsilon,\alpha}=0$)
 \begin{equation*}\begin{split} 2\partial_t\vv_{\eta,{\epsilon,\alpha}}\cdot&\vv_{\eta,{\epsilon,\alpha}}= \Delta( \vert \vv_{\eta,{\epsilon,\alpha}}\vert^2)- 2\vert \vN\otimes\vv_{\eta,{\epsilon,\alpha}}\vert^2\\ &-2\vv_{\eta,{\epsilon,\alpha}}\cdot\left(\mathbb{P}\Div((\varphi_\epsilon*(\theta_\alpha \vv_{\eta,{\epsilon,\alpha}}))\otimes\vb_{\eta,{\epsilon,\alpha}})+\mathbb{P}\Div ((\varphi_\epsilon*(\theta_\alpha \vb_{\eta,{\epsilon,\alpha}}))\otimes\vv_{\eta,{\epsilon,\alpha}})\right)\\ &-\Div( \vert  \vv_{\eta,{\epsilon,\alpha}}\vert^2 (\varphi_\epsilon*(\theta_\alpha \vv_{\eta,{\epsilon,\alpha}})  ))+2 \vert  \vv_{\eta,{\epsilon,\alpha}}\vert^2 \Div (\varphi_\epsilon*(\theta_\alpha \vv_{\eta,{\epsilon,\alpha}}) )\\&-2\Div(q_{\eta,\epsilon,\alpha} \vv_{\eta,{\epsilon,\alpha}})\end{split}\end{equation*} with
 \[\Div (\varphi_\epsilon*(\theta_\alpha \vv_{\eta,{\epsilon,\alpha}}) )= \varphi_\epsilon*(\vv_{\eta,{\epsilon,\alpha}}\cdot \vN\theta_\alpha).\]
 Integrating against $\Phi(x)^2\, dx$, we obtain
 \begin{equation*} \frac d{dt}\|\vv_{\eta,{\epsilon,\alpha}}\|_{L^2(\Phi_2\, dx)}^2 +2 \|\vN\otimes\vv_{\eta,{\epsilon,\alpha}}\|_{L^2(\Phi_2\, dx)}^2 =\sum_{k=1}^6 A_k
 \end{equation*} with:
 \\ $\bullet$ $A_1=\int \Phi^2  \Delta( \vert \vv_{\eta,{\epsilon,\alpha}}\vert^2)\, dx=\int  \vert \vv_{\eta,{\epsilon,\alpha}}\vert^2 \Delta(\Phi^2 )\, dx\leq C \|\vv_{\eta,{\epsilon,\alpha}}\|_{L^2(\Phi_2\, dx)}^2$ (since $\vert\Delta(\Phi)^2\leq C \Phi^2$);
  \\ $\bullet$ $A_2=-2\int \Phi^2  \vv_{\eta,{\epsilon,\alpha}}\cdot  \mathbb{P}\Div((\varphi_\epsilon*(\theta_\alpha \vv_{\eta,{\epsilon,\alpha}}))\otimes\vb_{\eta,{\epsilon,\alpha}}) \, dx$; since the Riesz transforms are bounded on $L^2(\Phi_2\, dx)$ and since $$\vert \int \Phi^2 f \partial_j g\, dx\vert\leq \|\Phi f\|_{H^1} \|\Phi \partial_j g\|_{H^{-1}}\leq C (\|\Phi f\|_2 +\|\Phi\vN f\|_2) \|\Phi g\|_2,$$ we have (since $\|\varphi_\epsilon*f\|_{L^2(\Phi_2\, dx)}\leq C \|\mathcal{M}_f\|_{L^2(\Phi_2\, dx)}\leq C' \|f\|_{L^2(\Phi_2\, dx)}$)
   \begin{equation*}\begin{split} A_2\leq& C (\|\vv_{\eta,{\epsilon,\alpha}}\|_{L^2(\Phi_2\, dx)} +\|\vN\otimes\vv_{\eta,{\epsilon,\alpha}}\|_{L^2(\Phi_2\, dx)}) \| (\varphi_\epsilon*(\theta_\alpha \vv_{\eta,{\epsilon,\alpha}}))\otimes\vb_{\eta,{\epsilon,\alpha}} \|_{L^2(\Phi_2\, dx)}
   \\ \leq &C'(\|\vv_{\eta,{\epsilon,\alpha}}\|_{L^2(\Phi_2\, dx)} +\|\vN\otimes\vv_{\eta,{\epsilon,\alpha}}\|_{L^2(\Phi_2\, dx)}) \| \vv_{\eta,{\epsilon,\alpha}}\|_{L^2(\Phi_2\, dx)} 
   \|\vb_{\eta,{\epsilon,\alpha}} \|_\infty \\\leq & C''(\|\vv_{\eta,{\epsilon,\alpha}}\|_{L^2(\Phi_2\, dx)} +\|\vN\otimes\vv_{\eta,{\epsilon,\alpha}}\|_{L^2(\Phi_2\, dx)})  \| \vv_{\eta,{\epsilon,\alpha}}\|_{L^2(\Phi_2\, dx)} 
 \eta t^{-\frac 3{2r}}
 \\ \leq & \frac 1{10}  \|\vN\otimes\vv_{\eta,{\epsilon,\alpha}}\|_{L^2(\Phi_2\, dx)}^2 + C''' \|\vv_{\eta,{\epsilon,\alpha}}\|_{L^2(\Phi_2\, dx)}^2 (1+ \eta^2 t^{-\frac 3 r}) ;
 \end{split}\end{equation*} 
   \\ $\bullet$ $A_3=-2\int \Phi^2  \vv_{\eta,{\epsilon,\alpha}}\cdot   \mathbb{P}\Div ((\varphi_\epsilon*(\theta_\alpha \vb_{\eta,{\epsilon,\alpha}}))\otimes\vv_{\eta,{\epsilon,\alpha}}) \, dx$; with similar computations as for $A_2$ we find  
    \begin{equation*}\begin{split} A_3\leq&  \frac 1{10}  \|\vN\otimes\vv_{\eta,{\epsilon,\alpha}}\|_{L^2(\Phi_2\, dx)}^2 + C''' \|\vv_{\eta,{\epsilon,\alpha}}\|_{L^2(\Phi_2\, dx)}^2 (1+ \eta^2 t^{-\frac 3 r}) ;    \end{split}\end{equation*} 
   $\bullet$ $A_4=-\int \Phi^2 \Div( \vert  \vv_{\eta,{\epsilon,\alpha}}\vert^2 (\varphi_\epsilon*(\theta_\alpha \vv_{\eta,{\epsilon,\alpha}})  ))\, dx$; since $\vert \vN \Phi^2\vert \leq C \Phi^3$, we have
   $$A_4\leq  C \int \Phi^3  \vert  \vv_{\eta,{\epsilon,\alpha}}\vert^2 \vert \varphi_\epsilon*(\theta_\alpha \vv_{\eta,{\epsilon,\alpha}})  \vert\, dx;$$ as $\Phi  \vv_{\eta,{\epsilon,\alpha}}\in L^2$, we have  $\Phi(\varphi_\epsilon*(\theta_\alpha \vv_{\eta,{\epsilon,\alpha}})) \in L^2$; on the other hand, we have $\Phi  \vv_{\eta,{\epsilon,\alpha}}\in L^2$, we have  $\Phi     \vv_{\eta,{\epsilon,\alpha}}\in H^1$, so that $\Phi     \vv_{\eta,{\epsilon,\alpha}}\in L^3\cap L^6$;  thus, we have
       \begin{equation*}\begin{split} A_4\leq& C \|\Phi  \vv_{\eta,{\epsilon,\alpha}}\|_6 \|\Phi  \vv_{\eta,{\epsilon,\alpha}}\|_3 \|\Phi ( \varphi_\epsilon*(\theta_\alpha \vv_{\eta,{\epsilon,\alpha}}) )\|_2
       \\ \leq& C'  \|\Phi  \vv_{\eta,{\epsilon,\alpha}}\|_6 \|\Phi  \vv_{\eta,{\epsilon,\alpha}}\|_3 \|\Phi  \vv_{\eta,{\epsilon,\alpha}}\|_2\\ \leq& C'' ( \| \vv_{\eta,{\epsilon,\alpha}}\|_{L^2(\Phi_2\, dx)} + \|\vN\otimes\vv_{\eta,{\epsilon,\alpha}}\|_{L^2(\Phi_2\, dx)})^{\frac 3 2}  \| \vv_{\eta,{\epsilon,\alpha}}\|_{L^2(\Phi_2\, dx)}^{\frac 3 2}\\ \leq &  \frac 1{10}  \|\vN\otimes\vv_{\eta,{\epsilon,\alpha}}\|_{L^2(\Phi_2\, dx)}^2 + C'''  \| \vv_{\eta,{\epsilon,\alpha}}\|_{L^2(\Phi_2\, dx)}^2 (1+  \| \vv_{\eta,{\epsilon,\alpha}}\|_{L^2(\Phi_2\, dx)}^4);
 \end{split}\end{equation*} 
 \\ $\bullet$ $A_5=-2\int\Phi^2 \Div(q_{\eta,\epsilon,\alpha} \vv_{\eta,{\epsilon,\alpha}})\, dx$; similarly, we have   $$A_5\leq  C \int \Phi^3  \vert q_{\eta,\epsilon,\alpha} \vert \vert  \vv_{\eta,{\epsilon,\alpha}}  \vert\, dx\leq  \|\Phi  \vv_{\eta,{\epsilon,\alpha}}\|_6 \|\Phi^2 q_{\eta,\epsilon,\alpha} \|_{6/5} ;$$ we recall that 
 \[ q_{\eta,\epsilon,\alpha}=\sum_{1\leq i\leq 3}\sum_{1\leq j\leq 3} R_iR_j ((\varphi_\epsilon*(\theta_\alpha \vv_{\eta,{\epsilon,\alpha},i}))\vv_{\eta,{\epsilon,\alpha},j}) \] with $\Phi  \vv_{\eta,{\epsilon,\alpha}}\in L^3$ and $\Phi (\varphi_\epsilon*(\theta_\alpha \vv_{\eta,{\epsilon,\alpha}}))\in L^2$;  thus, $$\Phi^2  ((\varphi_\epsilon*(\theta_\alpha \vv_{\eta,{\epsilon,\alpha}}))\otimes \vv_{\eta,{\epsilon,\alpha}})\in L^{\frac 6 5}$$ or, equivalently, $   (\varphi_\epsilon*(\theta_\alpha \vv_{\eta,{\epsilon,\alpha}}))\otimes \vv_{\eta,{\epsilon,\alpha}}\in L^{\frac 6 5}(\Phi^{\frac {12} 5}\, dx);$ as $\frac{12}5<3$, $\Phi^{\frac {12} 5}$ belongs to the Muckenhoupt class $\mathcal{A}_{\frac 6 5}$, and the Riesz transforms are bounded on $L^{\frac 6 5}(\Phi^{\frac {12} 5}\, dx)$; we thus get
      \begin{equation*}\begin{split} A_5\leq& C \|\Phi  \vv_{\eta,{\epsilon,\alpha}}\|_6 \|\Phi  \vv_{\eta,{\epsilon,\alpha}}\|_3 \|\Phi ( \varphi_\epsilon*(\theta_\alpha \vv_{\eta,{\epsilon,\alpha}}) )\|_2 \\ \leq &  \frac 1{10}  \|\vN\otimes\vv_{\eta,{\epsilon,\alpha}}\|_{L^2(\Phi_2\, dx)}^2 + C'  \| \vv_{\eta,{\epsilon,\alpha}}\|_{L^2(\Phi_2\, dx)}^2 (1+  \| \vv_{\eta,{\epsilon,\alpha}}\|_{L^2(\Phi_2\, dx)}^4);
 \end{split}\end{equation*} 
 \\ $\bullet$ $A_6=2\int \Phi^2   \vert  \vv_{\eta,{\epsilon,\alpha}}\vert^2 ( \varphi_\epsilon*(\vv_{\eta,{\epsilon,\alpha}}\cdot \vN\theta_\alpha))\, dx$; we have  
       \begin{equation*}\begin{split} A_6\leq& C \|\Phi  \vv_{\eta,{\epsilon,\alpha}}\|_6 \|\Phi  \vv_{\eta,{\epsilon,\alpha}}\|_3 \| \varphi_\epsilon*(  \vv_{\eta,{\epsilon,\alpha}}\cdot \vN \theta_\alpha )\|_2 \\ \leq & C \|\Phi  \vv_{\eta,{\epsilon,\alpha}}\|_6 \|\Phi  \vv_{\eta,{\epsilon,\alpha}}\|_3 \|  \vv_{\eta,{\epsilon,\alpha}}\cdot \vN \theta_\alpha \|_2; \end{split}\end{equation*}  as $\vert  \vN \theta_\alpha (x)\vert\leq C \max(\alpha,1)\Phi(x))$, we find
       $$A_6\leq  \frac 1{10}  \|\vN\otimes\vv_{\eta,{\epsilon,\alpha}}\|_{L^2(\Phi_2\, dx)}^2 + C'  \| \vv_{\eta,{\epsilon,\alpha}}\|_{L^2(\Phi_2\, dx)}^2 (1+  \max(\alpha,1)^6\| \vv_{\eta,{\epsilon,\alpha}}\|_{L^2(\Phi_2\, dx)}^4).$$

Thus, integrating on $(0,t)$ for $0<t<T$ with $T<\min(T_{\epsilon,\alpha}, T_{[\eta]})$, we obtain, for $C_T=\sup_{0<t<T} \| \vv_{\eta,{\epsilon,\alpha}}(t,.)\|_{L^2(\Phi_2\, dx)}^2$,
\[ C_T\leq \|\vv_{0,\eta}\|_{L^2(\Phi_2\, dx)}^2 + C T C_T+ C C_T T^{1-\frac 3 r}\eta^2 + C T \max(1,\alpha)^6 C_T^3. \]. We obtain $C_T\leq 4  \|\vv_{0,\eta}\|_{L^2(\Phi_2\, dx)}^2 $ if 
$$ C T (1+16 \max(1,\alpha)^6   \|\vv_{0,\eta}\|_{L^2(\Phi_2\, dx)}^4) < \frac 1 4$$ and
 \begin{equation*} C T^{1-\frac 3 r}\eta^2 <\frac 1 2.\tag*{\qedhere}\end{equation*}  
  \end{proof}

\subsection{Rescaling and global estimates} 
The minoration on $T_{\epsilon,\alpha}$ given in Proposition \ref{prop:mollif3} depends on $\eta$, i.e. on $\|\vb_{0,\eta}\|_r$ and on $\|\vv_{0,\eta}\|_{L^2(\Phi_2\, dx)}$. As a matter of fact, following  Fern\'andez-Dalgo \&  Lemari\'e-Rieusset \cite{FLR20} and Bradshaw, Kucavica \& Tsai \cite{BKT22}, we can partly get rid of this restriction:

\begin{theorem}\label{theo:mollif4}  Let $2< p<+\infty$ and $0<\gamma<2$.  For every $\vu_{0}\in L^p(\Phi_\gamma\, dx)$ with  $\Div \vu_0=0$,  let $\vu_{\epsilon,\alpha}$ be the (maximal) solution of the mollified equations   
  \begin{equation}\label{NSEq11} \left\{ \begin{split} \partial_t\vu_{\epsilon,\alpha}=&\Delta\vu_{\epsilon,\alpha} -\mathbb{P}\Div((\varphi_\epsilon*(\theta_{\alpha}\vu_{\epsilon}))\otimes\vu_{\epsilon,\alpha})
 \\ \lim_{t\rightarrow 0} \vu_{\epsilon,\alpha}(t,.)=&\vu_{0}\end{split}\right.\end{equation}   in $\mathcal{C}([0,T_\epsilon), L^p(\Phi_\gamma\, dx))$. For every $T>0$, there exists $\alpha_{T,\vu_0}>0$ such that, for every $0<\epsilon$ and every $0<\alpha<\alpha_{T,\vu_0}$, one has $T_{\epsilon,\alpha}>T$. 
 Moreover, 
 one may split $\vu_{\epsilon,\alpha}$ as $\vu_{\epsilon,\alpha}=\vb_{{\epsilon,\alpha},T}+\vv_{{\epsilon,\alpha},T}$ on $(0,T)\times\mathbb{R}^3$ with
\begin{itemize}
\item $\vb_{{\epsilon,\alpha},T}\in\mathcal{C}([0, T], L^2(\Phi_4\, dx))$,
 \item $\sup_{0\leq t\leq T} \|\vb_{{\epsilon,\alpha},T}(t,.)\|_r\leq C_{3,T, \vu_0}$,
 \item $\sup_{0< t\leq T} t^{\frac 1 2}\|\vN\otimes \vb_{{\epsilon,\alpha},T}(t,.)\|_r\leq  C_{3,T,\vu_0}$, 
 \item $\sup_{0< t\leq T} t^{\frac 3 {2r}}\| \vb_{{\epsilon,\alpha},T}(t,.)\|_\infty\leq  C_{3,T,\vu_0}$,
 \item $\vv_{{\epsilon,\alpha},T}\in\mathcal{C}([0, T], L^2(\Phi_2\, dx))$ with
\[\sup_{0\leq t\leq T} \|\vv_{{\epsilon,\alpha},T}(t,.)\|_{L^2(\Phi_2\, dx)}\leq  C_{3,T,\vu_0},\]
 \item $\vv_{\eta,{\epsilon,\alpha}}\in L^2((0, T), H^1(\Phi_2\, dx))$ with
\[\|\vv_{{\epsilon,\alpha},T}(t,.)\|_{L^2((0,T)(H^1(L^2(\Phi_2\, dx))}\leq   C_{3,T,\vu_0}\]  \end{itemize} where $C_{3,T,\vu_0}$ doesn't depend on $\epsilon$ nor on $\alpha$. 

In particular, we have
\begin{itemize}
\item  $\sup_{0<\epsilon, 0<\alpha<\alpha_T}\sup_{0<t<T}  \|\vu_{\epsilon,\alpha}(t,.)\|_{L^2(\Phi_{\frac 7 2}\, dx)}\leq  C_{4,T,\vu_0}$
\item $\sup_{0<\epsilon, 0<\alpha<\alpha_T} \|\vu_{\epsilon,\alpha}(t,.)\|_{L^2((0,T), H^{1/2}(\Phi_2\, dx))} \leq  C_{4,T,\vu_0}$
\item $\sup_{0<\epsilon, 0<\alpha<\alpha_T} \|\partial_t\vu_{\epsilon,\alpha}(t,.)\|_{L^2((0,T),H^{-4}(\Phi_8\, dx))}\leq  C_{4,T,\vu_0}$
\end{itemize} \end{theorem}
 
 \begin{proof} We first recall the results on $\vu_{\epsilon,\alpha}$ obtained in Propositions \ref{prop:mollif}, \ref{prop:mollif2}, \ref{prop:mollif3}. Let $T>0$. Following Proposition \ref{prop:mollif3}, we choose ${\eta_T}$ with ${\eta_T}\leq \frac 1{C_1(C_2T)^{\frac{r-3}{2r}}}$ 
  and we split $\vu_0$ into $\vu_0=\vb_{0,{\eta_T}}+\vv_{0,{\eta_T}}$ with $\|\vb_{0,{\eta_T}}\|_\infty<{\eta_T}$ and $\vv_{0,{\eta_T}}\in L^2(\Phi_2\, dx)$. We then split $\vu_{\epsilon,\alpha}$ into $\vu_{\epsilon,\alpha}=\vb_{{\eta_T},\epsilon,\alpha}+\vv_{{\eta_T},\epsilon,\alpha}$. We rescale $\vu_{\epsilon,\alpha}$ into  \begin{equation*}\begin{split}\vu_{\lambda,\epsilon,\alpha}(t,x)=\frac 1\lambda \vu_{\epsilon,\alpha}(\frac t{\lambda^2},\frac x t)=&\frac 1\lambda \vb_{{\eta_T},\epsilon,\alpha}(\frac t{\lambda^2},\frac x t)+\frac 1\lambda \vv_{{\eta_T},\epsilon,\alpha}(\frac t{\lambda^2},\frac x t)\\ =&\vb_{{\eta_T},\lambda,\epsilon,\alpha}(t,x)+\vv_{{\eta_T},\lambda,\epsilon,\alpha}(t,x).
 \end{split}\end{equation*}
 We write similarly
 $$ \vu_{\lambda,0}(x)=\frac 1 \lambda \vu_0(\frac x \lambda),\ \vv_{{\eta_T},\lambda,0}(x)=\frac 1 \lambda \vv_{0,{\eta_T}}(\frac x \lambda), \ \vb_{{\eta_T},\lambda,0}(x)=\frac 1 \lambda \vb_{0,{\eta_T}}(\frac x \lambda).  $$
 We have
 \begin{equation*}\begin{split} \partial_t \vu_{\lambda,\epsilon,\alpha}(t,x)=&\frac 1{\lambda^3} (\partial_t \vu_{\epsilon,\alpha})(\frac t{\lambda^2},\frac x \lambda) \\=& \frac 1{\lambda^3} (\Delta \vu_{\epsilon,\alpha}-\mathbb{P}\Div((\varphi_\epsilon*(\theta_\alpha \vu_{\epsilon,\alpha}))\otimes \vu_{\epsilon,\alpha})(\frac t{\lambda^2},\frac x \lambda)\\ =&\Delta\vu_{\lambda,\epsilon,\alpha}(t,x) -\mathbb{P}\Div(\frac 1 \lambda(\varphi_\epsilon*(\theta_\alpha \vu_{\epsilon,\alpha})(\frac t{\lambda^2},\frac x \lambda))\otimes \vu_{\lambda,\epsilon,\alpha}(t,x))
 \\=&\Delta\vu_{\lambda,\epsilon,\alpha}(t,x) -\mathbb{P}\Div( (\varphi_{\lambda \epsilon}*(\theta_{\frac\alpha \lambda} \vu_{\lambda,\epsilon,\alpha})(t,x))\otimes \vu_{\lambda,\epsilon,\alpha}(t,x))
 \end{split}\end{equation*}
 Thus, $ \vu_{\lambda,\epsilon,\alpha}$ is the solution of the mollified equation
 $$  \vu_{\lambda,\epsilon,\alpha}=e^{t\Delta}\vu_{\lambda,0}- B_{\lambda\epsilon,\frac\alpha\lambda}( \vu_{\lambda,\epsilon,\alpha}, \vu_{\lambda,\epsilon,\alpha}).$$ By Proposition  \ref{prop:mollif3}, we know that the existence time $T_{\lambda,\epsilon,\alpha}$ of $\vu_{\lambda,\epsilon,\alpha}$ can be controlled by below as
  $$ T_{\lambda, \epsilon,\alpha}\geq \min(\frac 1{C_2} \frac 1{(C_1\|\vb_{{\eta_T},\lambda,0}\|_r)^{\frac {2r}{r-3}}},  \frac 1{C_2} \frac 1{1+\max(1,\frac\alpha\lambda)^6\|\vv_{{\eta_T},\lambda,0}\|_{L^2(\Phi_2\, dx)}^4}).$$
  Thus, we have (since $\lambda^2 \|\vb_{{\eta_T},\lambda,0}\|_r^{\frac {2r}{r-3}}=\|\vb_{0,{\eta_T}}\|_r^{\frac {2r}{r-3}}$)
   \begin{equation*}\begin{split} T_{\epsilon,\alpha}=&\frac 1{\lambda^2} T_{\lambda, \epsilon,\alpha}\\ \geq &  \min(\frac 1{C_2} \frac 1{(C_1\|\vb_{0,{\eta_T}}\|_r)^{\frac {2r}{r-3}}},  \frac 1{C_2 \lambda^2} \frac 1{1+\max(1,\frac\alpha\lambda)^6\|\vv_{{\eta_T},\lambda,0}\|_{L^2(\Phi_2\, dx)}^4}).
   \end{split}\end{equation*}
   We have
   \[  \lambda^2( 1+\|\vv_{{\eta_T},\lambda,0}\|_{L^2(\Phi_2\, dx)}^4)=\lambda^2 +(\int \vert \vv_{0,{\eta_T}}(x)\vert^2 \frac 1{1+\vert x\vert^2} \frac{\lambda^2+\lambda^2 \vert x\vert^2}{1+\lambda^2 \vert x\vert^2}\, dx)^2.\]
   For every $x\in\mathbb{R}^3$, we have
   \[ \sup_{0<\lambda<1} \frac{\lambda^2+\lambda^2 \vert x\vert^2}{1+\lambda^2 \vert x\vert^2}=1\text{ and }\lim_{\lambda\rightarrow 0} \frac{\lambda^2+\lambda^2 \vert x\vert^2}{1+\lambda^2 \vert x\vert^2}=0.\]
   Thus, by dominated convergence, we get that
   \[ \lim_{\lambda\rightarrow 0}  \lambda^2( 1+\|\vv_{{\eta_T},\lambda,0}\|_{L^2(\Phi_2\, dx)}^4)=0.\] We take a $\lambda_T\in (01)$ such that  $\lambda_T^2( 1+\|\vv_{{\eta_T},\lambda_T,0}\|_{L^2(\Phi_2\, dx)}^4)<\frac 1{C_2T}$ and we get that $\displaystyle T_{\epsilon,\alpha}=\frac 1{\lambda_T^2} T_{\lambda_T, \epsilon,\alpha}\geq \max(1,\frac\alpha\lambda_T)^6 T$.
   
   In particular, we have   $T_{{\epsilon,\alpha}}>T$ provided that $ \alpha<\alpha_{T,\vu_0} =\frac 1{\lambda_T}$.   
   
   With those values of ${\eta_T}$ and $\lambda_T$, writing $\vb_{\epsilon,\alpha,T}=\vb_{{\eta_T},\epsilon,\alpha}$ and  $\vv_{\epsilon,\alpha,T}(t,x)=\vv_{{\eta_T},\epsilon,\alpha}(t,x)=\lambda_T \vv_{{\eta_T},\lambda_T,\epsilon,\alpha}(\lambda_T^2 t,\lambda_T x)$, we use Proposition \ref{prop:mollif2} and Proposition \ref{prop:mollif3} to get the following estimates on the solution $\vu_{\epsilon,\alpha}$ on $(0,T)$:
\begin{itemize} 
 \item $\sup_{0\leq t\leq T} \|\vb_{\epsilon,\alpha,T}(t,.)\|_r\leq 2\|\vb_{0,{\eta_T}}\|_r $,
 \item $\sup_{0< t\leq T} t^{\frac 1 2}\|\vN\otimes \vb_{{\epsilon,\alpha},T}(t,.)\|_r\leq  2  \|\vN W_1\|_1 \|\vb_{0,{\eta_T}}\|_r $, 
  \item $\sup_{0< t\leq T} t^{\frac 3 {2r}}\|   \vb_{{\epsilon,\alpha},T}(t,.)\|_\infty\leq  C_{3,T,\vu_0} \|\vb_{0,{\eta_T}}\|_r $, 
 \item $\sup_{0\leq t\leq T} \|\vv_{{\epsilon,\alpha},T}(t,.)\|_{L^2(\Phi_2\, dx)}\leq  \frac 2 {\lambda_T} \| \vv_{0,{\eta_T}}\|_{L^2(\Phi_2\, dx)}$,
 \item $ \|\vv_{{\epsilon,\alpha},T}(t,.)\|_{L^2((0,T)(H^1(L^2(\Phi_2\, dx))}\leq   \frac {C_2} {\lambda_T} \| \vv_{0,{\eta_T}}\|_{L^2(\Phi_2\, dx)}$.
\end{itemize}
(where we use   the inequality
   $$ \lambda \|f\|_{L^2(\phi_2\, dx)}^2\leq  \int  \frac 1{\lambda^2 } \vert f(\frac x \lambda)\vert^2 \Phi_2(x)\, dx=    \int   \vert f( x )\vert^2 \frac \lambda{1+\vert \lambda x\vert^2}  \, dx \leq \frac 1 \lambda \|f\|_{L^2(\phi_2\, dx)}^2$$ for $0<\lambda\leq 1$).

In particular, we have
\begin{itemize}
\item  $ \|\vu_{\epsilon,\alpha}(t,.)\|_{L^2(\Phi_{\frac 7 2}\, dx)}\leq  \|\vv_{\epsilon,\alpha,T}(t,.)\|_{L^2(\Phi_{ 2}\, dx)}+ \|\vb_{\epsilon,\alpha,T}(t,.)\|_r)$, so that $$\sup_{0<\epsilon, 0<\alpha<\alpha_T}\sup_{0<t<T}  \|\vu_{\epsilon,\alpha}(t,.)\|_{L^2(\Phi_{\frac 7 2}\, dx)}\leq  C_{4,T,\vu_0}$$
\item We have \begin{equation*}\begin{split}\|\Phi^2 \vb_{\epsilon,\alpha,T}\|_{H^{1/2}} \leq &\sqrt{\|\Phi^2 \vb_{\epsilon,\alpha,T}\|_{2} \|\Phi^2 \vb_{\epsilon,\alpha,T}\|_{H^{1}} } \\ \leq\sqrt{ \|\vb_{\epsilon,\alpha,T}\|_{L^2(\Phi_4\, dx)}}&\sqrt{  \|\vb_{\epsilon,\alpha,T}\|_{L^2(\Phi_4\, dx)}+ \|\vN\otimes \vb_{\epsilon,\alpha,T}\|_{L^2(\Phi_4\, dx)})}\\ \leq &C\sqrt{ \|\vb_{\epsilon,\alpha,T}\|_{r}}\sqrt{ ( \|\vb_{\epsilon,\alpha,T}\|_{r}+ \|\vN\otimes \vb_{\epsilon,\alpha,T}\|_r}
\\\leq &C' (1+\sqrt{C_{3,T,\vu_0}} t^{-\frac 3 {4r}} )\|\vb_{0,\eta_T}\|_r.
\end{split}\end{equation*} 
As \begin{equation*}\begin{split} \|\vu_{\epsilon,\alpha}(t,.)&\|_{L^2((0,T), H^{1/2}(\Phi_4\, dx))}\\ \leq&  \|\vv_{\epsilon,\alpha}(t,.)\|_{L^2((0,T), H^{1}(\Phi_2\, dx))} + \|\vb_{\epsilon,\alpha}(t,.)\|_{L^2((0,T), H^{1/2}(\Phi_4\, dx))} ,\end{split}\end{equation*} 
we get  that $$\sup_{0<\epsilon, 0<\alpha<\alpha_T} \|\vu_{\epsilon,\alpha}(t,.)\|_{L^2((0,T), H^{1/2}(\Phi_2\, dx))} \leq  C_{4,T,\vu_0}$$
\item   From the proof of Theorem \ref{theo:stokes}, we see that $\mathbb{P}\Div ((\varphi_\epsilon*(\theta_\alpha \vu_{{\epsilon,\alpha}}))\otimes \vu_{{\epsilon,\alpha}})$ and $\Delta  \vu_{ {\epsilon,\alpha}}$ are bounded in $L^2 ((0,T), H^{-4}(\Phi_8\, dx))$. Thus, we have   $$\sup_{0<\epsilon, 0<\alpha<\alpha_T} \|\partial_t\vu_{\epsilon,\alpha}(t,.)\|_{L^2((0,T),H^{-4}(\Phi_8\, dx))}\leq  C_{4,T,\vu_0}.$$
\end{itemize} 
Theorem \ref{theo:mollif4} is proved.
 \end{proof}

\section{Solutions to the Navier\ddh Stokes equations.}
In this final section, we prove Theorem \ref{theo:weightLp}. The key tool for going from mollified equations to Navier\ddh Stokes equations will be the following lemma (a simpler variant of the Aubin\ddh Lions theorem):

\begin{lemma} [Rellich\ddh Lions lemma]
\label{le:rellich} Let $\sigma<0<s$, $0<T<+\infty$ and $R>0$. Let $(u_n)_{n\in\mathbb{N}}$ be a sequence of functions on $(0,T)\times\mathbb{R}^3$ such that $u_n$ is bounded in  $ L^2((0,T), H^s)$ and $\partial_t u_n$ is bounded in $L^2((0,T),H^\sigma)$ and $u_n(t,x)=0$ for $\vert x\vert>R$.  Then there exists a subsequence  $(u_{n_k})_{k\in\mathbb{N}}$  and $u_\infty \in  L^2((0,T)\times\mathbb{R}^3 $ such that $u_{n_k}$ is strongly convergent to $u_\infty$  in $L^2((0,T), L^2)$.
\end{lemma} 

\begin{proof} Extend $u_n$ to $\mathbb{R}\times \mathbb{R}^3$ in the following way: let $\omega\in \mathcal{D}(\mathbb{R})$ with $\omega(t)=1$ for $\vert t-\frac T 2\vert \leq   \frac{3T}4$ and $\omega(t)=0$ for $\vert t-\frac T 2\vert \geq   \frac{7T}8$. Define $v_n$ 
on $\mathbb{R}\times \mathbb{R}^3$ by $v_n(t,x)=u_n(t,x)$ if $0\leq t\leq T$, $=\omega(t) u_n(-t,x)$ if $-T\leq t\leq 0$, $=u_n(2T-t) $ if $T\leq t\leq 2T$ and $=0$ if $t\notin [-T,2T]$.  Then $v_n$ is bounded in $L^2(\mathbb{R}, H^s)$ and $\partial_t v_n$ is bounded in $L^2(\mathbb{R}, H^{\sigma})$. Taking the Fourier transfom on $\mathbb{R}\times \mathbb{R}^3$ defined by
$$ \check{F}(\tau,\xi)=\iint F(t,x) e^{-i(t\tau+x\cdot\xi)}\, dt\, dx,$$ we find that
$ (1+\vert \xi\vert^2)^{s/2} \check{v}_n$ is bounded in  $L^2( \mathbb{R}\times \mathbb{R}^3)$
and
$\tau  (1+\vert \xi\vert^2)^{\sigma/2} \check{v}_n$ is bounded in  $L^2( \mathbb{R}\times \mathbb{R}^3)$. As $\sigma<s$, 
$(1+\vert \tau\vert^2)^{1/2} (1+\vert \xi\vert^2)^{\sigma/2} \check{v}_n$ is bounded in  $L^2( \mathbb{R}\times \mathbb{R}^3)$.  Let  $  s_0=\frac s{1-\sigma+ s}$; we have $0<s_0<1$ and $s_0=(1-s_0)s+s_0\sigma$, so that
       \begin{equation*}\begin{split} (1+\tau^2+\vert \xi\vert^2)^{s_0/2}\leq& (1+\tau^2)^{s_0/2}(1+\vert \xi\vert^2)^{s_0/2}\\=& \left((1+\tau^2)^{1/2}(1+\vert \xi\vert^2)^{\sigma/2}\right)^{s_0}  \left(1+\vert \xi\vert^2)^{s/2}\right)^{1-s_0}
 \end{split}\end{equation*} and thus 
$(1+\tau^2+\vert \xi\vert^2)^{s_0/2}\check{v}_n$ is bounded in  $L^2( \mathbb{R}\times \mathbb{R}^3)$, or equivalently $v_n$ is bounded in $H^{s_0}( \mathbb{R}\times \mathbb{R}^3)$. As $s_0>0$ and as all the $v_n$ are supported in the compact set $[-T/2,2T]\times \overline{B(0,R)}$, we may apply Rellich's theorem and get that there exists a subsequence $(v_{n_k})_{k\in\mathbb{N}}$ that is strongly convergent in $L^2( \mathbb{R}\times \mathbb{R}^3)$. The subsequence $(u_{n_k})_{k\in\mathbb{N}}$   is then strongly convergent in $L^2((0,T), L^2( \mathbb{R}^3))$. 
   \end{proof}
   
   Theorem 1 will then be proved in the following way:
   
   \begin{proposition}\label{prop:final} There exists a sequence $(\epsilon_n,\alpha_n)_{n\in\mathbb{N}}$ and a vector field $\vu$ such that:
   \begin{itemize}
   \item $\lim_{n\rightarrow +\infty} \epsilon_n=\lim_{n\rightarrow +\infty} \alpha_n=0$,
   \item $\vu_{\epsilon_n, \alpha_n}$ is strongly convergent to $\vu$ in $L^2((0,T), \Phi_4\, dx)$ fo every $T>0$.
   \end{itemize}
   Moreover, $\vu$ is  a solution on   $(0,+\infty)\times\mathbb{R}^3$  of the Navier\ddh Stokes equations
  \begin{equation}\label{NSEq12} \left\{ \begin{split} \partial_t\vu=&\Delta\vu -\mathbb{P}\Div(\vu\otimes\vu))
\\\Div\,\vu=&0  \\ \lim_{t\rightarrow 0} \vu(t,.)=&\vu_0\end{split}\right.\end{equation} and, for every $0<T<+\infty$, we have
$\vu\in L^\infty((0,T), L^2(\frac{dx}{1+\vert x\vert^2}))+ L^\infty((0,T), L^r)$.
   \end{proposition}

\begin{proof} Let $\mathbb{N}^2$ be enumerated as $ \{ (j_n,k_n)\ /\ n\geq 1\ \}$. We start with a sequence $(\epsilon_{0,n},\alpha_{0,n})_{n\in\mathbb{N}}$  which converges to $(0,0)$. We choose  sequences $(\epsilon_{q,n},\alpha_{q,n})_{n\in\mathbb{N}}$ by induction on $q$, such that the sequence $(\epsilon_{q+1,n},\alpha_{q+1,n})_{n\in\mathbb{N}}$ will be a subsequence of the sequence $(\epsilon_{q,n},\alpha_{q,n})_{n\in\mathbb{N}}$.

Assume that we have chosen the sequence $(\epsilon_{q,n},\alpha_{q,n})_{n\in\mathbb{N}}$ for some $q\geq 0$. We take
 $\omega_{q+1}\in \mathcal{D}(\mathbb{R}^3)$ such that $ \omega_{q+1}(x)=1$ for $\vert x\vert\leq R_{q+1}=2^{k_{q+1}}$ and $=0$ for $\vert x\vert>2 R_{q+1}$. Let $T_{q+1}=2^{j_{q+1}}$.  For $\alpha_{q,n}< \alpha_{T_{q+1},\vu_0}$, $\omega_{q+1}\vu_{\epsilon_{q,n},\alpha_{q,n}}$ is bounded in $L^2((0,T_{q+1}), H^{1/2})$ and  supported in $\{x\in\mathbb{R}^3\ /\ \vert x\vert\leq 2R_{q+1}\},$ while $\partial_t(\omega_{q+1}\vu_{\epsilon_{q,n},\alpha_{q,n}})$ is bounded in $L^2((0,T_{q+1}), H^{-4})$. We may apply Lemma \ref{le:rellich} and choose a subsequence  $(\epsilon_{q+1,n},\alpha_{q+1,n})_{n\in\mathbb{N}}$ such that $\omega_{q+1}\vu_{\epsilon_{q+1,n},\alpha_{q+1,n}}$ is strongly convergent in $L^2((0,T_{q+1}), L^2)$ and thus $\vu_{\epsilon_{q+1,n},\alpha_{q+1,n}}$ is strongly convergent in $L^2((0,T_{q+1})\times B_{R_{q+1}})$.
 
 We then use Cantor's diagonal argument and define $(\epsilon_n,\alpha_n)=(\epsilon_{n,n},\alpha_{n,n})$. We have
 
   \begin{itemize}
   \item $\lim_{n\rightarrow +\infty} \epsilon_n=\lim_{n\rightarrow +\infty} \alpha_n=0$,
   \item $\vu_{\epsilon_n, \alpha_n}$ is strongly convergent   in $L^2((0,T), L^2(B_R))$ for every $T>0$ and every $R>0$.
   \end{itemize}
   With similar arguments, we can grant that, for any $T_k$, we have that $\vu_{\epsilon_n,\alpha_n}=\vv_{\epsilon_n,\alpha_n,T_k} +\vb_{\epsilon_n,\alpha_n,T_k} $ on $(0,T_k)\times\mathbb{R}^3$ with the strong convergence of $\vv_{\epsilon_n,\alpha_n,T_k}$ to  $\vv_{T_k}$ and of  $\vb_{\epsilon_n,\alpha_n,T_k}$ to  $\vb_{T_k}$   in $L^2((0,T_k), L^2(B_R))$ for  every $R>0$.

   Let $\vu$ be the limit of  $\vu_{\epsilon_n, \alpha_n}$. We have, for $R>0$ and $\chi_R$ the characteristic function of the ball $B_R$ (and for $\alpha_n<\alpha_{T_k,\vu_0}$),
      \begin{equation*}\begin{split}[\bullet] \|\vb_{T_k}&-\vb_{\epsilon_n, \alpha_n,T_k}\|_{L^2((0,T_k), L^2(\Phi_4\, dx))} \\\leq& \frac 1 {  R^{\frac 1 4}} \|\vb_{T_k}-\vb_{\epsilon_n, \alpha_n,T_k}\|_{L^2((0,T_k),  L^2(\Phi_{7/2}\, dx))}+\|\chi_R(\vb_{T_k}-\vb_{\epsilon_n, \alpha_n,T_k})\|_{L^2((0,T_k), L^2)} \\\leq& \frac C {  R^{\frac 1 4}} \|\vb_{T_k}-\vb_{\epsilon_n, \alpha_n,T_k}\|_{L^2((0,T_k), L^r))}+\|\vb_{T_k}-\vb_{\epsilon_n, \alpha_n,T_k}\|_{L^2((0,T_k)\times B_R)}. \end{split}\end{equation*} As $ \|\vb_{T_k}-\vb_{\epsilon_n, \alpha_n,T_k}\|_{L^\infty((0,T_k), L^r))} \leq C_{T_k,\vu_0}$, we find that
      $$\limsup_{n\rightarrow +\infty} \|\vb_{T_k}-\vb_{\epsilon_n, \alpha_n,T_k}\|_{L^2((0,T_k), L^2(\Phi_4\, dx))} \leq C C_{T_k,\vu_0} \sqrt{T_k}  \frac 1 {  R^{\frac 1 4}}. $$ Letting $R$ go to $+\infty$, we find 
      $\lim_{n\rightarrow +\infty} \|\vb_{T_k}-\vb_{\epsilon_n, \alpha_n,T_k}\|_{L^2((0,T_k), L^2(\Phi_4\, dx))} =0$.
      \\ $ [\bullet] $ Similarly, we have   \begin{equation*}\begin{split} \|\vv_{T_k}&-\vv_{\epsilon_n, \alpha_n,T_k}\|_{L^2((0,T_k), L^2(\Phi_4\, dx))} \\\leq& \frac 1 {  R} \|\vv_{T_k}-\vv_{\epsilon_n, \alpha_n,T_k}\|_{L^2((0,T_k), L^2(\Phi_{2}\, dx))}+\|\chi_R(\vv_{T_k}-\vv_{\epsilon_n, \alpha_n,T_k})\|_{L^2((0,T_k), L^2)}. \end{split}\end{equation*} As $ \|\vv_{T_k}-\vv_{\epsilon_n, \alpha_n,T_k}\|_{L^\infty((0,T_k), L^2(\Phi_2\, dx))} \leq C_{T_k,\vu_0}$,   we find 
      $$\lim_{n\rightarrow +\infty} \|\vv_{T_k}-\vv_{\epsilon_n, \alpha_n,T_k}\|_{L^2((0,T_k),L^2(\Phi_4\, dx))} =0.$$
    \\ $ [\bullet] $ We remark that we have the inequalities  \\ $\| \varphi_{\epsilon_n} *(\theta_{\alpha_n} f)\|_r\leq \|f\|_r$, \\ $\| \varphi_{\epsilon_n} *(\theta_{\alpha_n} f)\|_{L^2(\Phi_2\, dx)}\leq \|\mathcal{M}_f\|_{L^2(\Phi_2\, dx)}\leq C \|f\|_{L^2(\Phi_2\, dx)}$ (as $\Phi_2$ is a Muckenhoupt weight), \\ $\| \varphi_{\epsilon_n} *(\theta_{\alpha_n} f)\|_{L^2(B_R)} =\| \varphi_{\epsilon_n} *(\theta_{\alpha_n} \chi_{R+1}f)\|_{L^2(B_R)}  \leq \|f\|_{L^2(B_{R+1})}$ (for $\epsilon_n<1$),\\ so that we have 
      $$\lim_{n\rightarrow +\infty} \| \varphi_{\epsilon_n} *(\theta_{\alpha_n}(\vb_{T_k}-\vb_{\epsilon_n, \alpha_n,T_k}))\|_{L^2((0,T_k), L^2(\Phi_4\, dx))} =0$$ and 
      $$\lim_{n\rightarrow +\infty} \| \varphi_{\epsilon_n} *(\theta_{\alpha_n}(\vv_{T_k}-\vv_{\epsilon_n, \alpha_n,T_k}))\|_{L^2((0,T_k), L^2(\Phi_4\, dx))} =0.$$    \\ $ [\bullet] $ Similarly, we have the inequalities   \\ $\|f- \varphi_{\epsilon_n} *(\theta_{\alpha_n} f)\|_r=\|\varphi_{\epsilon_n} *((1-\theta_{\alpha_n}) f)\|_r\leq \|f\|_r$, \\ $\| f-\varphi_{\epsilon_n} *(\theta_{\alpha_n} f)\|_{L^2(\Phi_2\, dx)}\leq \|\mathcal{M}_{(1-\theta_{\alpha_n})f}\|_{L^2(\Phi_2\, dx)}\leq C \|f\|_{L^2(\Phi_2\, dx)}$, \\ $\|f- \varphi_{\epsilon_n} *(\theta_{\alpha_n} f)\|_{L^2(B_R)} =\| \varphi_{\epsilon_n} *((1-\theta_{\alpha_n}) \chi_{R+1}f)\|_{L^2(B_R)}  \leq \|(1-\theta_{\alpha_n}) \chi_{R+1}f\|_{2}$ (for $\epsilon_n<1$).\\  As  we have $(1-\theta_{\alpha_n}) \chi_{R+1}=0$ for $\alpha_n<\frac 1{R+1}$, we find that
      $$\lim_{n\rightarrow +\infty} \| \vb_{T_k}-\varphi_{\epsilon_n} *(\theta_{\alpha_n}\vb_{T_k} )\|_{L^2((0,T_k), L^2(\Phi_4\, dx))} =0$$ and 
      $$\lim_{n\rightarrow +\infty} \| \vv_{T_k}-\varphi_{\epsilon_n} *(\theta_{\alpha_n}\vv_{T_k} )\|_{L^2((0,T_k), L^2(\Phi_4\, dx))} =0.$$ 
      
      Combining all those estimates, we find that for every $T>0$
      $$\lim_{n\rightarrow +\infty} \|\vu\otimes \vu-(\varphi_{\epsilon_n} *(\theta_{\alpha_n}\vu_{\epsilon_n,\alpha_n}))\otimes  \vu_{\epsilon_n,\alpha_n})\|_{L^1((0,T), L^2(\Phi_4\, dx))} =0.$$ From the proof of Theorem \ref{theo:stokes}, we see that $$\lim_{n\rightarrow +\infty}\Delta \vu_{\epsilon_n,\alpha_n} - \mathbb{P}\Div ((\varphi_\epsilon*(\theta_\alpha \vu_{\epsilon_n,\alpha_n}))\otimes \vu_{\epsilon_n,\alpha_n})=\Delta \vu-\mathbb{P}\Div(\vu\otimes\vu)$$  in $L^1((0,T), H^{-4}(\Phi_8\, dx))$. $\vu$ is a solution of the Navier\ddh Stokes equations (\ref{NSEq12}). \end{proof}

\end{document}